\documentclass{article}
\usepackage{amsmath}
\usepackage{amsthm}
\usepackage{graphics}

\newtheorem{theorem}{Theorem}[section]
\newtheorem{cor}[theorem]{Corollary}
\newtheorem{lem}[theorem]{Lemma}
\newtheorem{prop}[theorem]{Proposition}

\theoremstyle{definition}

\newtheorem{rem}[theorem]{Remark}

\newtheorem{example}[theorem]{Example}

\numberwithin{equation}{section}

\usepackage{amssymb}
\usepackage{amsfonts}
\usepackage{psfrag}
\usepackage{epsfig}
\usepackage{url}
\usepackage{mathtools}

\newcommand{\E}{\mathbb{E}}

\newcommand{\R}{\mathbb{R}}

\newcommand{\cO}{\mathcal{O}}

\begin{document}

\title{Best finite constrained approximations\\of one-dimensional probabilities}

 \author{Chuang Xu\\Department of Mathematical
   Sciences\\University of Copenhagen\\Copenhagen, {\sc
     Denmark}\\[2mm] and \\[2mm]
Arno Berger\\Mathematical and Statistical
   Sciences\\University of Alberta\\Edmonton, Alberta, {\sc
     Canada}}

\maketitle

\begin{abstract}
\noindent
This paper studies best finitely supported approximations of
one-dimensional probability measures with respect to the
$L^r$-Kantorovich (or transport) distance, where either the
locations or the weights of the approximations' atoms are
prescribed. Necessary and sufficient optimality conditions are
established, and the rate of convergence (as the number of atoms goes
to infinity) is discussed. In view of emerging mathematical and
statistical applications, special attention is given to the case of
best uniform approximations (i.e., all atoms having equal weight). The
approach developed in this paper is elementary; it is based on best 
approximations of (monotone) $L^r$-functions by step functions, 
and thus different from, yet naturally complementary to, the classical 
Voronoi partition approach.
\end{abstract}

\noindent
\hspace*{8.3mm}{\small {\bf Keywords.} Constrained approximation, 
best uniform approximation, asymptotically best}\\ 
\hspace*{26.2mm} {\small  approximation, Kantorovich distance, balanced function, quantile function.}

\noindent
\hspace*{8.3mm}{\small {\bf MSC2010.} 28A33, 60B10, 62E17.}

\section{Introduction}\label{sec1}

Finding best finitely supported approximations of a given (Borel)
probability measure $\mu$ on $\R$ is an important basic problem
that has been studied extensively and from several
perspectives. Assuming for instance that $\int_{\R} |x|^r {\rm
  d}\mu(x)<+\infty$ for some $r\ge 1$, a classical question asks to
minimize the $L^r$-Kantorovich (or transport) distance $d_r(\nu, \mu)$ over all discrete probabilities
$\nu$ supported on at most $n$ atoms, where $n$ is a given positive integer. A rich theory of {\em
  quantization of probability measures\/} addresses this question, as
well as applications thereof in such diverse
fields as information theory, numerical integration, and optimal
transport, among others; see, e.g., \cite{BJM, GL,P} and the many
references therein. As is well known, a minimal value
of $d_r(\nu, \mu)$ always is attained for some discrete probability
$\nu = \delta_{\bullet}^{\bullet,n}$ which may or may not be
determined uniquely by this minimality property. Moreover, $d_r
(\delta_{\bullet}^{\bullet,n}, \mu )\to 0$ as $n\to \infty$, and
the precise rate of convergence has attracted
particular interest. A celebrated theorem (see Proposition \ref{Zador}
below) asserts that, under a mild moment condition,
$\bigl(n d_r(\delta_{\bullet}^{\bullet,n}, \mu ) \bigr)$ converges to a
finite positive limit whenever $\mu$ is non-singular (w.r.t.\ Lebesgue
measure). Results in a similar spirit have been
established for important classes of singular measures, notably
self-similar and -conformal probabilities \cite{GL1,KZ,PB1}. While these
classical results crucially employ Voronoi partitions (as
developed in some detail, e.g., in \cite{GL}), alternative tools and
extensions to other metrics have recently been studied also
\cite{BJM,CGI,DV}.

A second important perspective on the approximation problem is that of
{\em random empirical quantization\/} \cite{BG,DSS}. To illustrate it, let $(X_j)_{j\ge
  1}$ be an iid.\ sequence of random variables with common law $\mu$,
and consider the (random) empirical measure $\mu_n = \frac1{n}
\sum_{j=1}^n \delta_{X_j}$; here and throughout, $\delta_a$ is a Dirac
unit mass at $a\in \R$. Then $d_r (\mu_n,
\mu )\to 0$ with probability one as $n\to \infty$, as well as $\E
d_r (\mu_n, \mu ) \to  0$. A comprehensive analysis of the rate of
convergence of $\bigl( \E d_r  (\mu_n , \mu ) \bigr)$ is provided by
the recent monograph \cite{BL} which, in particular, identifies
necessary and sufficient conditions for decay to occur at the
``standard rate'' $ (n^{-1/2} )$, that is, for $ \bigl( n^{1/2} \E
d_r ( \mu_n , \mu ) \bigr)$ to be bounded above and below by finite
positive constants. Beyond these one-dimensional
results, rates of convergence for random empirical quantization have
lately been studied in higher dimensions and other settings also; see,
e.g., \cite{BG,DSS,FG}.

The purpose of the present article is to develop a third perspective
on the approximation problem that in a sense lies between the two
established perspectives briefly recalled above. Specifically, we
present an in-depth study of finitely supported approximations that
are {\em non-random\/} yet {\em constrained\/} in that either the
locations or the weights of the approximations' atoms are
prescribed. To the best of our knowledge, such approximations have
not been studied systematically in the literature, though the recent
papers \cite{Ba} and \cite{BJM} do consider (uniform)
``$\mathcal{U}$-quantization'' and discrete approximations of
absolutely continuous probabilities $\mu$, respectively.
The necessary and sufficient conditions for best constrained
approximations presented in this article make no assumptions on $\mu$ beyond $\int_{\R} |x|^r {\rm
  d}\mu(x)<+\infty$. They follow rather directly from elementary
properties of monotone functions and exploit a certain duality between
locations and weights of atoms. By contrast, note that Voronoi partitions
are typically much less useful if weights, rather than positions, are
prescribed \cite{GL}.

Arguably the simplest special case where our
results apply is that of best {\em uniform\/} approximations: Given
$\mu$ and a positive integer $n$, for which $\nu = \frac1{n}
\sum_{j=1}^n \delta_{x_{j}}$ is $d_r  (\nu, \mu )$ minimal, where $x_{1},
\cdots ,  x_{n}\in \R$? Theorem \ref{th1} below characterizes the (often
unique) minimizer $\delta_{\bullet}^{{\bf u}_n}$; here usage of the
superscript ${\bf u}_n$ emphasizes the fact that a {\em best uniform
  approximation\/} $\delta_{\bullet}^{{\bf u}_n}$ typically is quite
different from any {\em best approximation\/}
$\delta_{\bullet}^{\bullet, n}$. The special case of best uniform
approximations is of considerable interest in itself: In statistics, when dealing with
empirical data sets, practical considerations may
demand that all atoms have equal weights, or at least that they be integer
multiples of one fixed unit weight \cite{BHM}. Also, best uniform
approximations are close analogues of {\em support points\/}
\cite{MJ}, the latter being minimizers relative to a slightly
different metric (energy distance). One may thus view
$\delta_{\bullet}^{{\bf u}_n}$ as a quasi Monte Carlo (MC) tool that
minimizes the integration error bound $|\int f \, {\rm d}\mu - \sum_{j=1}^n
f(x_{j})|\le \mbox{\rm Lip}\,  (f) \, d_1 (\delta_{\bullet}^{{\bf u}_n}, \mu)$ for a wide class of functions (cf.\ \cite{DP}), and
consequently a careful analysis of $d_1(\delta_{\bullet}^{{\bf u}_n}, \mu)$ as
$n\to \infty$ is indispensable. In computational
mathematics, best uniform approximations also arise naturally in the
form of restricted MC methods, and a basic question is how
performance of the latter compares to general (non-restricted) MC
methods, that is, how $d_r (\delta_{\bullet}^{{\bf u}_n}, \mu)$
compares to $d_r (\delta_{\bullet}^{\bullet, n}, \mu)$ as $n\to
\infty$; see, e.g., \cite{GHMR, G18} and the many references
therein. Restricted MC methods have recently found applications in
``big-data'' problems in Bayesian statistics \cite{MJ} and the
numerical solution of SDE, notably in mathematical finance \cite{GHMR,
  Glass, PPP}.

Just as for the best unconstrained and the random empirical approximations mentioned
earlier, $d_r (\delta_{\bullet}^{{\bf u}_n}, \mu)\to 0$ as $n\to \infty$,
which again makes the rate of convergence a natural object of
study. Presented in Subsection \ref{subwei}, our results in this regard are
quite similar to those of \cite{BL}, despite their obviously different
context. As a simple illustrative example, consider the
standard normal distribution, i.e., let $\frac{ {\rm d }}{{\rm
    d}x}\mu(\, ] -\infty , x]
  )=\frac{1}{\sqrt{2\pi}}e^{-x^2/2}$ for all $x\in\R$. From Proposition \ref{Zador} below, it follows that in
the case of best (unconstrained) approximations, for all $r\ge 1$,
$$
d_r(\delta_{\bullet}^{\bullet, n}, \mu) = \cO ( n^{-1} ) \quad \mbox{\rm as
} n \to \infty \, ,
$$
whereas for random empirical approximations, \cite[Sec.6.5]{BL}
shows that
$$
\E d_r (\mu_n , \mu ) = \left\{
\begin{array}{ll}
\cO ( n^{-1/2}) & \mbox{\rm if } 1\le r < 2 \, , \\[1mm]
\cO ( n^{-1/2} (\log\log n)^{1/2} ) & \mbox{\rm if } r= 2  \, , \\[1mm]
\cO ( n^{-1/r}(\log n)^{1/2})  & \mbox{\rm if } r> 2 \, .
\end{array}
\right.
$$
By contrast, for best uniform approximations, with $r=2$ and
along the subsequence $n=2^k$, the sharp rate of convergence for
$\bigl(d_r(\delta_{\bullet}^{{\bf u}_n},\mu)\bigr)$ is $(n\log n)^{-1/2}$,
as proved by \cite{GHMR}. Utilizing the main results of the present
article, notably Theorem \ref{th1} below, one can
show that in fact
$$
d_r (\delta_{\bullet}^{{\bf u}_n} , \mu ) = \left\{
\begin{array}{ll}
\cO ( n^{-1 } (\log n)^{1/2})  & \mbox{\rm if } r= 1  \, , \\[1mm]
\cO ( n^{-1/r}(\log n)^{-1/2})  & \mbox{\rm if } r> 1 \, .
\end{array}
\right.
$$
Not too surprisingly, the rate of convergence of $\bigl(
d_r (\delta_{\bullet}^{{\bf u}_n} , \mu ) \bigr)$ is slower than that of
$\bigl( d_r (\delta_{\bullet}^{\bullet, n} , \mu ) \bigr)$, but
faster than that of $\bigl( \E d_r (\mu_n , \mu )\bigr)$.

Due to the nature of the underlying approximation problem for monotone functions,
our approach is not restricted to $d_r$, and results in a similar
spirit can be established for other important
metrics and for discrete approximations with countable support. 
One-dimensionality, on the other hand, is crucial: In
multi-dimensional (Euclidean) spaces, upper bounds for the rate of
decay of best uniform approximations have been established only
recently \cite{C}, via a {\em uniform decomposition approach}. In
addition, we mention \cite{GHMR} which analyzes a best uniform approximation
problem (referred to as {\em random bit quadrature}) in a Hilbert
space setting with $L^2$-Kantorovich metric, motivated also by MC
applications.

This article is organized as follows. Section \ref{No} introduces the
notations used throughout, and recalls definition and basic
properties of the metric $d_r$ for the reader's convenience. Section
\ref{Mono} reviews several elementary facts about monotone functions and
their quantile and growth sets, as well as the notion of a balanced
function, to be used subsequently in Section \ref{Lr} to characterize best
approximations of (monotone) $L^r$-functions by step functions. While
they may also be of independent interest, these results crucially
serve as tools in Section \ref{sec5}, the main part of this work. In that
section, necessary and sufficient conditions for best finite
approximations with prescribed locations (Subsection \ref{subloc}) or
weights (Subsection \ref{subwei}) are established. Much attention
is devoted to the special case of best uniform approximations
$\delta_{\bullet}^{{\bf u}_n}$, and in particular to the rate of convergence
of $\bigl( d_r (\delta_{\bullet}^{{\bf u}_n}, \mu)\bigr)$. Convergence
theorems and finite range (upper and lower) bounds for
such sequences are provided. All results are illustrated via simple
examples of $\mu$ which include absolutely continuous
(exponential, Beta) as well as singular (Cantor, inverse Cantor)
probability measures.

\section{Notations}\label{No}
The following, mostly standard notations are used throughout. The
natural and real numbers are denoted $\mathbb{N}$ and $\mathbb{R}$,
respectively. The extended real numbers are
$\overline{\mathbb{R}}=\mathbb{R}\cup\{-\infty,+\infty\}$. 
For any $a\in\overline{\mathbb{R}}$, ${\rm sgn}\ a=1$ if $a>0$, ${\rm
  sgn}\ 0=0,$  and ${\rm sgn}\ a=-1$ if $a<0$. The indicator function
of any set $A\subset\overline{\mathbb{R}}$ is denoted ${\bf 1}_A$, and $\log$ symbolizes the natural logarithm.
For $x\in\mathbb{R}$, let $|x|$ be the absolute value, $\lfloor
x\rfloor$ the floor (i.e., the largest integer $\le x$), and $\langle
x\rangle =x-\lfloor x\rfloor$ the fractional part of $x$,
respectively. Lebesgue measure on $\overline{\mathbb{R}}$ is
symbolized by $\lambda$, and $\delta_a$ stands for the Dirac measure
concentrated at $a$, i.e., $\delta_a(A)={\bf 1}_{A}(a)$ for all $A$.

The usual notations for intervals, e.g., $[a,b[\ =
\{x\in\overline{\mathbb{R}}: a\le x<b \}$ are used. Endowed with the
topology $ \{[-\infty,a[\ \cup\ U\ \cup\ ]b,+\infty]:
a,b\in\overline{\mathbb{R}},\ U\subset \mathbb{R}\ \text{open} \}$,
the space $\overline{\mathbb{R}}$ is compact and homeomorphic to the
unit interval $\mathbb{I}=[0,1]$. Throughout,
$I\subset\overline{\mathbb{R}}$ always denotes a closed (and hence
compact) interval that is non-degenerate, i.e., $\lambda(I)>0$.
For $A\subset\overline{\mathbb{R}}$, denote by $\#\ A,$
$\overset{\circ}{A}$, and $\overline{A}$ the cardinality (number of
elements), interior, and closure of $A$, respectively. 
Every non-empty $A$ has an infimum $\inf A$ and a supremum $\sup A$; 
if $A$ is closed, then $\inf A=\min A$ and $\sup A=\max A$. 
If $A\subset\overline{\mathbb{R}}$ is an interval and $f:
A\to\overline{\mathbb{R}}$ is monotone, then
$f(a-)=\lim_{\varepsilon\downarrow0}f(a-\varepsilon)$ and
$f(a+)=\lim_{\varepsilon\downarrow0}f(a+\varepsilon)$ both exist for
every $a\in\overset{\circ}{A}$. For any set
$A\subset\overline{\mathbb{R}}$ and any function $f:
A\rightarrow\overline{\mathbb{R}}$, the image of $A$ under $f$ is
$f(A)= \{f(a): a\in A \}$, while the pre-image of
$B\subset\overline{\mathbb{R}}$ is $f^{-1}(B)= \{a\in A: f(a)\in B
\}$. Also, for every $b\in\overline{\mathbb{R}}$, 
let $ \{f\le b \}=f^{-1} ([-\infty,b] )$; the sets $ \{f\ge b \}$, $
\{f<b \}$, $ \{f>b \}$, and $ \{f=b \}$ are defined analogously.
Denote by ${\rm essinf}_A f$ and ${\rm esssup}_A f$  the essential
infimum and supremum of $f$ on $A$, respectively. For $1\le r<+\infty$ 
and any (closed, non-degenerate) interval $I\subset\overline{\mathbb{R}}$, let
$L^r(I)$ be the space of all measurable functions
$f:I\to\overline{\mathbb{R}}$ that are (absolutely) $r$-integrable
with respect to $\lambda,$ and $L^{\infty}(I)$ the space of all
functions bounded $\lambda$-almost everywhere (a.e.). For  $f\in
L^r(I)$, let $f^+=\max \{f,0 \}$ and $f^-= (-f )^+$, hence
$f=f^+-f^-$.

Let $\mathcal{P}$ be the family of all Borel probability measures on 
$\overline{\mathbb{R}}$ with $\mu(\mathbb{R})=1$. For every
$\mu\in\mathcal{P}$,  $F_{\mu}:\overline{\mathbb{R}}\to\mathbb{I}$
with $F_{\mu}(x)=\mu ([-\infty,x] )$ is the associated {\em
  distribution function}, 
$F_{\mu}^{-1}$ the associated {\em (upper) quantile function}, i.e., 
\begin{equation} \label{00}
F_{\mu}^{-1}(t)=\sup \{F_{\mu}\le t \},\quad  \forall\ t\in]0,1[\, ,
\end{equation} 
and ${\rm supp}\ \mu$ the {\em support} of $\mu$, 
that is, the smallest closed set of $\mu$-measure $1$. 
Both $F_{\mu}$ and $F^{-1}_{\mu}$ are non-decreasing and
right-continuous. As a consequence, $F_{\mu}^{-1}$ generates a 
positive Borel measure $\mu^{-1}$ on $]0,1[$ via 
$$
\mu^{-1} (]t,u] )=F_{\mu}^{-1}(u)-F_{\mu}^{-1}(t),\quad  \forall 0<t<u<1\, .
$$ 
Note that $\mu^{-1}$, referred to as the {\em inverse measure} of
$\mu$, is finite if and only if ${\rm supp}\ \mu$ is bounded, since 
in fact $\mu^{-1} (]0,1[ )=\max {\rm supp}\ \mu-\min {\rm supp}\ \mu$; 
see, e.g., \cite[App.A]{BL} for further basic properties of inverse measures.

For every $r\ge1,$ the set of probability measures with finite $r$-th
moment is denoted $\mathcal{P}_r$, i.e., 
$\mathcal{P}_r= \bigl \{\mu\in\mathcal{P}:  \int_\mathbb{R}|x|^r{\rm
  d}\mu(x)<+\infty \bigr\}$. 
Thus $\mu\in\mathcal{P}_r$ if and only if $F_{\mu}^{-1}\in
L^r(\mathbb{I})$. On $\mathcal{P}_r$, the {\em $L^r$-Kantorovich
  distance} $d_r$ is 
\begin{equation}\label{0}
   d_r(\mu,\nu)=\left(\int_{\mathbb{I}}\left|F_{\mu}^{-1}(t)-F_{\nu}^{-1}(t)
   \right|^r{\rm
     d}t\right)^{1/r}=\left\|F_{\mu}^{-1}-F_{\nu}^{-1}\right\|_r,\quad
 \forall \mu,\nu \in\mathcal{P}_r\, .
 \end{equation} 
For $r=1$, by Fubini's
 theorem,
$$
d_1(\mu,\nu)=\int_{\mathbb{R}}\left|F_{\mu}(x)-F_{\nu}(x)\right|{\rm
   d}x,\ \quad \forall\ \mu,\nu\in\mathcal{P}_1\, .
$$
When endowed with the metric $d_r$, the space $\mathcal{P}_r$ is separable and
 complete, and $d_r(\mu_n,\mu) \to 0$ implies that  $\mu_n\to\mu$
 weakly. Note that $\mathcal{P}_r\supset\mathcal{P}_s$ and $d_r\le
 d_s$ whenever $r< s$. On $\mathcal{P}_s$, the metrics $d_r$ and
 $d_s$ are not equivalent, as the example of
 $\mu_n=(1-n^{-s})\delta_0+n^{-s}\delta_n$ shows, for which
 $d_s(\mu_n,\delta_0)\equiv1$, and yet
 $\lim_{n\to\infty}d_r(\mu_n,\delta_0)=0$ for all $r<s$, and hence
 $\mu_n\to\delta_0$ weakly. The reader is referred to \cite{D,RR} for
 details on the mathematical background of the Kantorovich distance,
 and to \cite{GL,RR} for a discussion of its usefulness in the study of mass
 transportation and quantization problems.

\section{Monotone and balanced functions and their inverses}\label{Mono}

Quantization, as informally alluded to in the Introduction, may be
understood as the approximation of a given probability measure by 
finite weighted sums of point masses. Every quantile function is
non-decreasing; in particular, the quantile function associated with 
a finitely supported probability measure is a monotone step function. 
Therefore, it is natural---not least in view of \eqref{0}---to
formulate the ensuing approximation problem more generally as a 
problem about the best approximation of monotone $L^r$-functions by
step functions. Towards this goal, we first present some relevant
properties of monotone functions. For ease of exposition, the focus is 
on non-{\em decreasing} functions, but all subsequent arguments hold 
analogously for non-increasing functions as well.

Given an interval $I\subset\overline{\mathbb{R}}$ and a non-decreasing
function $f: I\to \overline{\mathbb{R}},$ define the
$t$-\textit{quantile set\/} $Q_t^f$ of $f$ as
$$
Q_t^f= [\inf \{f\ge t \},\sup \{f\le t \} ],\quad  \forall
t\in\overline{\mathbb{R}} \, ;
$$
here and throughout, $\inf\varnothing:=\max I$ and
$\sup\varnothing:=\min I$. Also remember that $I$ is closed and non-degenerate, by convention.
As a generalization of \eqref{00}, the {\em (upper) inverse} function $f^{-1}: \overline{\mathbb{R}}\to\overline{\mathbb{R}}$ associated with $f$ is
$$
f^{-1}(t):=\sup \{f\le t \}=\max Q_t^f,\quad
\forall t \in \overline{\mathbb{R}}\, . 
$$
Note that $f^{-1}$ is non-decreasing, right-continuous and, on $f(I)$, 
coincides with the ordinary inverse of $f$ whenever $f$ is
one-to-one. Moreover, $ (f^{-1} )^{-1} (x)=f(x+)$ for all
$x\in\overset{\circ}{I}$; in particular, therefore, $ (f^{-1} )^{-1}$
equals $f$ a.e. on $\overset{\circ}{I}$, and in fact everywhere if $f$
is right-continuous. A few elementary properties of quantile sets are as follows.

\begin{prop}\label{le1}{\rm \cite[Lem.~2.7]{BHM}}.
Let $f: I\to\overline{\mathbb{R}}$ be non-decreasing. Then, for every
$t\in\overline{\mathbb{R}}$, the set $Q_t^f$ is a non-empty, 
compact (possibly one-point) subinterval of $I,$ and $f(x)=t$ whenever
$\min Q_t^f<x<\max Q^f_t$; in particular, $Q_t^f $ equals $ \overline{ \{ f =
  t\}}$ whenever the latter set is non-empty. Moreover, the following
hold:
\begin{enumerate}
\item If $t<u$ then $x\le y$ for every $x\in Q_t^f$ and every $y\in Q_u^f,$  and the set $Q_t^f\cap Q_u^f$ contains at most one point.
\item For every $x\in I$ and $t\in\overline{\mathbb{R}}$, $x\in Q_t^f$
  if and only if $t\in Q_x^{f^{-1}}$.
\end{enumerate}
\end{prop}

\noindent
For any non-decreasing function $f: I\to\overline{\mathbb{R}}$, call
$x\in I$ a {\em growth point\/} of $f$ if $f(y)<f(x)$ for all $y\in I$
with $y<x$, or $f(y)>f(x)$ for all $y>x$; see also \cite[p.97]{BL}. 
Define the {\em growth set\/} of $f$ as
$$
G^f= \{x\in I: x\ \text{is\ a\ growth\ point\ of}\ f \}\, .
$$
Thus for example, $G^{F_{\mu}}={\rm supp}\ \mu$ for every
$\mu\in\mathcal{P}$, and $ \{0,1 \}\subset
G^{F_{\mu}^{-1}}\subset\mathbb{I}$. 
An elementary relation between growth and quantile sets follows directly from the definitions.

\begin{prop}\label{pro2}
Let $f: I\to\overline{\mathbb{R}}$ be non-decreasing. Then $G^f$ is a
closed subset of $I$, non-empty unless $f$ is constant, and
$G^f\cup\  \{\min I,\max I \}=\bigcup_{t\in\overline{\mathbb{R}}}  \{\min Q_t^f,\max Q_t^f \}$.
\end{prop}

Next, we recall a useful terminology from \cite{CM}:
Given a bounded interval $I\subset\mathbb{R},$ call a measurable function $f:\ I\to\overline{\mathbb{R}}$ {\em balanced\/} if
$$
\big|\lambda ( \{f>0 \} )-\lambda ( \{f<0 \} )\big|\le\lambda
 ( \{f=0 \} )\, ,
$$
and denote by $B^f:= \{t\in\mathbb{R}: f-t\ \text{is\ balanced} \}$  the {\em set of all balanced values\/} of $f$.
To establish a few basic properties of $B^f$ (in Lemma~\ref{le5}
below), consider the  auxiliary function
$\ell_f:\overline{\mathbb{R}}\to\overline{\mathbb{R}}$  
given by 
$$\ell_f(t)={\textstyle \frac{1}{2}} \bigl( \min I+\max I+\lambda ( \{f<t \} )-\lambda (
\{f>t \} ) \bigr) \, .
$$
The following properties of $\ell_f$ are straightforward to verify.

\begin{prop}\label{pro0}
Let $I$ be a bounded interval and $f: I\to\overline{\mathbb{R}}$ a
measurable function. Assume that $f$ is finite a.e.. Then the
following hold:
\begin{enumerate}
\item $\ell_f$ is non-decreasing;
\item For every $t\in\mathbb{R}$,
  $\ell_f(t\pm)=\ell_f(t)\pm\frac{1}{2}\lambda\left(\left\{f=t\right\}\right)$,
  and hence $\ell_f$ is continuous at $t$ if and only if
  $\lambda\left(\left\{f=t\right\}\right)=0$. Moreover,
  $\lambda(\{\ell_f^{-1}<t \}\cap
  I)=\lambda\left(\left\{f<t\right\}\right)$ as well as
  $\lambda(\{\ell_f^{-1}>t\}\cap I )=\lambda\left(\{f>t\}\right)$;
\item  $\lim_{t\to-\infty}\ell_f(t)=\ell_f(-\infty)=\min I$ and $\lim_{t\to+\infty}\ell_f(t)=\ell_f(+\infty)=\max I$;
\item If $f$ is non-decreasing
  then 
$$\ell_f(t)={\textstyle \frac{1}{2}}
\left(f^{-1}(t)+f^{-1}(t-)\right),\quad 
  \forall\ t\in\mathbb{R}\, , 
$$
and also 
$$
\ell_f^{-1}(x)= (f^{-1} )^{-1}(x)=f(x+),\quad \ell_f^{-1}(x-)=f(x-),\quad
\forall\ x\in\overset{\circ}{I}\, ;
$$
\item If $f\in L^r(I)$ for some $1\le r<+\infty$, then
  $\big\|\ell_f^{-1}-t\big\|_r= \|f-t \|_r$ for every
  $t\in\mathbb{R}$.
\end{enumerate}
\end{prop}

\begin{example}\label{ex0}
Let $I=\mathbb{I}$ and 
$$
f(x) = \left\{
\begin{array}{ll}
\frac34 (x+1) & \mbox{\rm if } 0\le x < \frac13 ,\\[2mm]
\frac12  & \mbox{\rm if }\frac13  \le  x < \frac23  ,\\[2mm]
\frac32 x - 1  & \mbox{\rm if } \frac23 \le x \le 1.
\end{array}
\right.
$$
Here the functions $\ell_f,\
\ell_f^{-1}:\overline{\mathbb{R}}\to\overline{\mathbb{R}}$ are given
by
$$
\ell_f(t) =\left\{
\begin{array}{ll}
\frac23 \max \{t,0\} & \mbox{\rm if } t<\frac12 , \\[2mm]
\frac12  & \mbox{\rm if } t= \frac12  , \\[2mm]
\frac23  & \mbox{\rm if } \frac12< t <\frac34  , \\[2mm]
\frac43 \min \{t,1\} - \frac13 & \mbox{\rm if } t\ge \frac34 ,
\end{array}
\right.
\quad
\ell_f^{-1}(x) =\left\{
\begin{array}{ll}
-\infty & \mbox{\rm if } x<0 , \\[2mm]
\frac12 \min \{ 3x, 1\}   & \mbox{\rm if } 0 \le x <\frac23  , \\[2mm]
\frac14 (3x+1) & \mbox{\rm if } \frac23 \le x <  1  , \\[2mm]
+\infty & \mbox{\rm if } x\ge 1 .
\end{array}
\right.
$$
Figure \ref{f0} illustrates that indeed
$\big\|\ell_f^{-1}-t\big\|_r= \|f-t \|_r$ for all $t$, as asserted by
Proposition~\ref{pro0}(v).
\end{example}

\begin{figure}[h] 
\psfrag{tv1}[r]{\small $1$}
\psfrag{tv12}[r]{\small $\frac12$}
\psfrag{tv13}[r]{\small $\frac13$}
\psfrag{tv23}[r]{\small $\frac23$}
\psfrag{tv0}[r]{\small $0$}
\psfrag{th1}[]{\small $1$}
\psfrag{th13}[]{\small $\frac13$}
\psfrag{th12}[]{\small $\frac12$}
\psfrag{th23}[]{\small $\frac23$}
\psfrag{th0}[]{\small $0$}
\psfrag{tx}[]{$x$}
\psfrag{tt}[]{$t$}
\psfrag{tlt}[]{$\ell_f (t)$}
\psfrag{tlix}[]{$\ell_f^{-1} (x)$}
\psfrag{tfx}[]{$f(x)$}
\begin{center}
\includegraphics{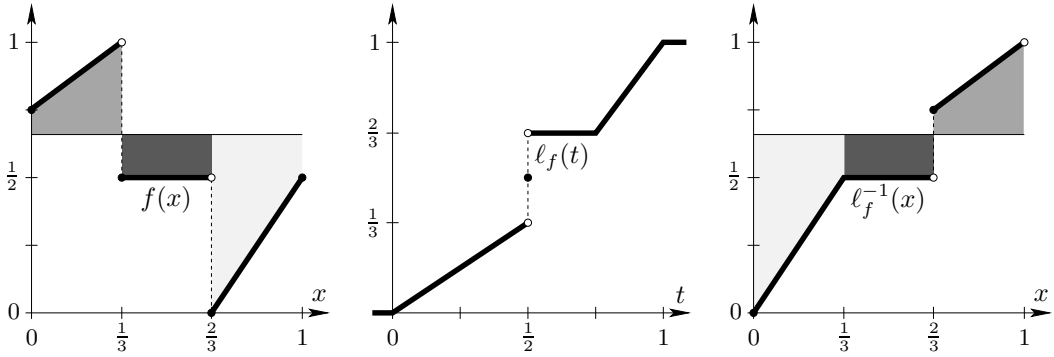}
\end{center}
\caption{Graphing $f$, $\ell_f$, and $\ell_f^{-1}$ of Example \ref{ex0}.}\label{f0}
\end{figure}

\begin{rem}
By Proposition~\ref{pro0}(v), minimizing $t\mapsto \|f-t \|_r$ for
$f\in L^r(I)$ is equivalent to minimizing
$t\mapsto\big\|\ell_f^{-1}-t\big\|_r$ for the {\em monotone\/} 
function $\ell^{-1}_f.$ Note also that if $f\in L^r(I)$ is non-decreasing then $f$ and $\ell^{-1}_f$ coincide a.e.,
by Proposition~\ref{pro0}(iv).
\end{rem}

Utilizing Propositions~\ref{le1} and~\ref{pro0}, we now establish a
few basic properties of the sets $B^f$ that will be used in the next
section.

\begin{lem}\label{le5}
Let $I$ be a bounded interval and $f:I\to\overline{\mathbb{R}}$ a
measurable function. Assume that $f$ is finite a.e.. Then
$B^f=Q_{\frac{1}{2}(\min I+\max I)}^{\ell_f}.$ Moreover, the following
hold:
\begin{enumerate}
\item For every $t\in\mathbb{R},$
  $\lambda\left(\left\{f>t\right\}\right)>\lambda\left(\left\{f\le
      t\right\}\right)$ if $t<\min B^f$, and
  $\lambda\left(\left\{f<t\right\}\right)>\lambda\left(\left\{f\ge
      t\right\}\right)$ if $t>\max B^f$;
\item $\lambda \bigl(f^{-1} (\overset{\circ}{B^f} ) \bigr) =0$;
\item $\lambda\left(\left\{f\le\min B^f\right\}\right)=\lambda\left(\left\{f\ge\max B^f\right\}\right)$.
\end{enumerate}
\end{lem}

\begin{proof}
For convenience, let  $\xi=\frac{1}{2}(\min I+\max I)$, and note that,
by definition, $B^f= \{t:  |\ell_f(t)-\xi |\le\frac{1}{2}\lambda (
\{f=t \} ) \}$. Define $a=\inf \{\ell_f\ge \xi \},\ b=\sup \{\ell_f\le
\xi \}$, and hence $[a,b]=Q^{\ell_f}_{\xi}.$ It is easy to see that $a$ and $b$ are finite, with $a\le b$, and
$$
\ell_f(a-)\le \xi\le \ell_f(a+),\quad  \ell_f(b-)\le \xi\le \ell_f(b+)\, , 
$$
which implies  that $a, b\in B^f$, by Proposition~\ref{pro0}(ii). 
For every $t\in \, ]a,b[$, $\ell_f(t)=\xi$, thus $t\in B^f$, and hence $[a,b]\subset B^f$.
For every $t>b$, $\ell_f(t-)>\xi$, so again by
Proposition~\ref{pro0}(ii), $\ell_f(t)-\xi>\frac{1}{2}\lambda ( \{f=t
\} )$, which implies that $t\notin B^f$ and  $\lambda ( \{f<t \}
)>\lambda ( \{f\ge t \} )$. 
Similarly, $t\notin B^f$ and $\lambda ( \{f>t \} )>\lambda ( \{f\le t
\} )$ for every $t<a$. This proves that $B^f=[a,b]=Q_{\xi}^{\ell_f}$, 
and also establishes (i).

To prove (ii) and (iii), assume that
$\overset{\circ}{B^f}\neq\varnothing$, i.e., $a<b$. For every
$t\in\overset{\circ}{B^f}$, $\ell_f(t)=\xi,$ i.e., $\lambda ( \{f>t \}
)=\lambda ( \{f<t \} )$. Hence for all $a<t_1<t_2<b$, 
$$
\lambda ( \{f<t_1 \} )\le\lambda ( \{f<t_2 \} )=\lambda
 ( \{f>t_2 \} )\le\lambda ( \{f>t_1 \} )=\lambda ( \{f<t_1
 \} )\, .
$$ 
Thus $\lambda ( \{t_1\le f<t_2 \} )=\lambda (\{t_1<f\le t_2 \} )=0$. 
Letting $t_1\downarrow a$ and $t_2\uparrow b$, properties (ii) and (iii) immediately follow from the continuity of $\lambda$.
\end{proof}

\begin{rem}
If, under the assumptions of Lemma~\ref{le5}, the function $f$ is non-decreasing, then $B^f=Q^{f^{-1}}_{\frac{1}{2}(\min I+\max I)}$.
\end{rem}

\section{Approximating $L^r$-functions by step functions}\label{Lr}

This section characterizes the best approximations of a given function
by step functions. Two main results (Lemma~\ref{le3} and
Theorem~\ref{le2}) will be used in Section 5 to identify best finitely
supported approximations of a given probability measure
$\mu\in\mathcal{P}$; they may also be of independent interest.
Throughout this section, we assume that the closed, non-degenerate
interval $I\subset\mathbb{R}$ is {\em bounded}. (For unbounded $I$, 
most statements become either trivial or meaningless.)

First, we give a result on the best approximation of a monotone
function by a (monotone) step function with a prescribed range and a
single jump at a variable location.

\begin{lem}\label{le3}
Assume that $f:I\to\overline{\mathbb{R}}$ is non-decreasing, and $f\in
L^r(I)$ for some $r\ge1$. Let $a, b\in\mathbb{R}$ with $a<b$. Then the value of 
$$
\left\|f-\left(a{\bf 1}_{ [\min I,\xi
        [}+b{\bf 1}_{ [\xi,\max I ]}\right)\right\|_r,\quad  \forall \xi\in I\, , $$ 
is minimal if and only if $\xi\in Q_{\frac{1}{2}(a+b)}^f$.
\end{lem}

\begin{proof}
Given $f\in L^r(I)$ and $a<b$, define
$\psi(\xi)=\big\|f-\left(a{\bf 1}_{ [\min I,\xi [}+b{\bf 1}_{
    [\xi,\max I ]}\right)\big\|_r$ for all $\xi\in I$, and let
$c=\frac{1}{2}(a+b)$. Clearly, the function $\psi$ is non-negative and
continuous, and so attains a minimal value.
If $\xi>f^{-1}(c)$ then there exists $0<\varepsilon<\xi-f^{-1}(c)$ such that
$f(x)>c$ for\ all $x\in [\xi-\varepsilon,\xi ]$. Hence
\begin{align*}
      \psi(\xi)^r-\psi \bigl( f^{-1}(c) \bigr)^r & =\int_{f^{-1}(c)}^{\xi} ( |f(x)-
      a |^r- |f(x)-b |^r )\, {\rm d}x\\
      & = \int_{f^{-1}(c)}^{\xi} \bigl ( (f(x)-a )^r- (f(x)-b )^r
      \bigr)\, 
      {\bf 1}_{ \{f\ge b \}}{\rm d}x\\
&\quad +\int_{f^{-1}(c)}^{\xi} \bigl( (f(x)-a )^r-\bigl (b-f(x)
\bigr)^r \bigr) \, {\bf 1}_
{ \{f<b \}}{\rm d}x\\
 & \ge \int_{\xi-\varepsilon}^{\xi}(b-a)^r{\bf 1}_{ \{f\ge b \}}{\rm d}x+\int_{\xi-\varepsilon}^{\xi} (2f(x)-a-b )^r{\bf 1}_{ \{f<b \}}
{\rm d}x\\
& \ge \ \varepsilon\min \{b-a,2(f(\xi-\varepsilon)-c) \}^r>0\, , 
\end{align*}
i.e., $\psi(\xi)>\psi \bigl( f^{-1}(c) \bigr)$. Similarly, $\psi(\xi)>\psi
\bigl(f^{-1}(c) \bigr)$ whenever $\xi<\inf\{f\ge c\}$. Therefore $\psi$ 
attains its minimal value on the interval $\left[\inf\{f\ge
  c\},f^{-1}(c)\right]=Q_c^f$, 
and the proof will be complete once it is shown that in fact $\psi$ is
constant on $Q_c^f$. If $Q_c^f$ is a singleton, then, trivially, this
is the case. On the other hand, if $\xi, \eta\in
\overset{\circ}{Q^f_c}$ with $\xi<\eta$, then $f ( [\xi,\eta ])=\{c\}$, by Proposition~\ref{le1}, and 
$$
\psi(\eta)^r-\psi(\xi)^r=\int_{\xi}^{\eta} ( |f(x)-b |^r- |f(x)-
      a |^r )\, {\rm d}x
      =\int_{\xi}^{\eta} (|c-b|^r-|c-a|^r )\, {\rm d}x=0\, .
$$ 
Thus $\psi$ is constant on $Q_c^f$, as claimed.
\end{proof}

\begin{rem}\label{re1}
The monotonicity of $f$ is essential in Lemma~\ref{le3}. To see this,
take for instance $I=[0,5]$ and the (non-monotone) function $f=16\cdot{\bf 1}_{[0,1[}+8\cdot{\bf 1}_{[1,2[}+18\cdot{\bf 1}_
{[2,3[}+9\cdot{\bf 1}_{[3,5]}$. For $a=0,\ b=24$, it is
straightforward to verify that $\big\|f-24\cdot{\bf 1}_{ [\xi,5
  ]}\big\|_r$ is minimal precisely for $\xi\in  \{0,2,5 \}$ if $r=1$
or $r=2$, for $\xi=5$ if $1<r<2$, and for $\xi\in \{0,2 \}$ if $r>2$. 
In general, therefore, the set of minimizers $\xi$ is not an interval and may depend on $r$.
\end{rem}

The remainder of this section deals with a problem that is dual to the one
addressed by Lemma~\ref{le3}, namely the best approximation of an
$L^r$-function $f$ by a step function with prescribed locations but
variable jumps. By considering intervals of constancy individually,
clearly it is enough to consider the approximation of $f$ by a {\em
  constant\/} function. Remember that the closed, non-degenerate
interval $I\subset\mathbb{R}$ is assumed to be bounded throughout.

\begin{theorem}\label{le2}
Assume that $f\in L^{r_0}(I)$ for some $r_0\ge1$. 
Then for every $1\le r\le r_0$, there exists $\tau_r^f\in\mathbb{R}$
such that 
$$\left\|f-\tau_r^f\right\|_r\le \|f-t \|_r,\quad \forall
t\in\mathbb{R}.
$$ 
Moreover, the following hold:
\begin{enumerate}
\item $\tau_r^f\in [{\rm essinf}_I f, {\rm esssup}_I f ]$;
\item $ \|f-t \|_1=\left\|f-\tau_1^f\right\|_1$ if and only if $t\in
  B^f$;
\item For $1<r\le r_0$, the number $\tau_r^f$ is unique, and $r\mapsto \tau_r^f$ is continuous.
\end{enumerate}
\end{theorem}

\begin{proof}
Given $f\in L^{r_0}(I)$, recall that $f\in L^r(I)$ for every $1\le
r\le r_0$, since $I$ is bounded. Hence the auxiliary function $\phi_r$
given by 
\begin{equation}\label{phi}
\phi_r(t)=\lambda(I)^{-1/r} \|f-t \|_r,\quad \forall  t\in\mathbb{R}\, ,
\end{equation} 
is well defined and real-valued. Note that
$\lim_{|t|\to+\infty}\phi_r(t)=+\infty$. Since $\phi_r$ is convex,
there exists $\tau_r^f\in\mathbb{R}$ such that $\phi_r (\tau_r^f )\le\phi_r(t)$ for all $t\in\mathbb{R}$.

It remains to prove assertions (i)--(iii).
To establish (i), let $b={\rm esssup}_I f$ for convenience, and
observe that, for all  $t>b$, 
\begin{align*}
    \lambda(I) (\phi_r(t)^r-\phi_r(b)^r )=\int_I \Bigl( \bigl( t-f(x)\bigr)^r-\bigl(b-f(x)\bigr)^r
    \Bigr){\rm d}x& \ge\int_I(t-b)^r{\rm d}x\\
& =\lambda(I)(t-b)^r>0,
\end{align*} 
hence $\phi_r(t)>\phi_r(b)$. Similarly, $\phi_r(t)>\phi_r ({\rm
  essinf}_I f )$ whenever $t<{\rm essinf}_I f$. This shows that 
$\tau_r^f\in [{\rm essinf}_I f, {\rm esssup}_I f ]$.

To prove (ii), given $t>\max B^f$, pick any $u$ with $\max B^f<u<t$.
Then, 
\begin{align*}
\lambda(I) \bigl(\phi_1(t)  -\phi_1(u) \bigr) & =\int_I (|f(x)-t|-|f(x) -u|)\, {\rm d}x\\
& \ge \int_{ \{f<u \}} \! \Bigl (t-f(x)- \bigl(u-f(x)\bigr) \Bigr) \, {\rm
  d}x+\int_{ \{f\ge u \}} \! \bigl(f(x)-t- (f(x)-u) \bigr) \, {\rm d}x\\
& = \ (t-u)\bigl( \lambda(\{f<u\})-\lambda(\{ f\ge u\})\bigr) > 0\, ,
\end{align*}
by Lemma~\ref{le5}(i), and so $\tau_1^f\le\max B^f$. Similarly, 
$\tau_1^f\ge\min B^f$. On the other hand, if $t, u\in B^f$, then
\begin{align*}
      \lambda(I) \bigl(\phi_1(u)-\phi_1(t) \bigr) & = \int_{ \{f\le\min
        B^f \}} \Bigl(u-f(x)- \bigl(t-f(x) \bigr) \Bigr){\rm d}x\\ & \quad  +\int_{f^{-1} (\overset{\circ}{B^f} )} ( |f(x)-u |- |f(x)
      -t | ){\rm d}x\\ &\quad  +\int_{ \{f\ge\max B^f \}} \bigl(f(x)-u- (f(x)-t )
     \bigr ){\rm d}x\\ & = \ (u-t) \bigl( \lambda ( \{f\le\min B^f \} )-\lambda (
      \{f\ge\max B^f \} ) \bigr)=0\, ,
\end{align*}
by Lemma~\ref{le5}(ii) and (iii). Thus $\phi_1(t)$ is minimal if and only if $t\in B^f$.

Regarding (iii), we claim that the number $\tau_r^f$ is unique for
$1<r\le r_0$. Trivially, this is true if $f$ is essentially
constant. In any other case, note that $\phi_r^r$ is differentiable
w.r.t. $t$, and
\begin{equation}\label{phid}
\begin{split}\frac{\lambda(I)}{r}\frac{{\rm d}\phi^r_r(t)}{{\rm d}t} &
  = \int_I |f(x)-t |^{r-1}{\rm sgn} \bigl(t-f(x) \bigr)\, {\rm d}x\\
& = \int_{ \{f<t \}} \bigl(t-f(x) \bigr)^{r-1}{\rm d}x-\int_{ \{f>t \}} (f(x)-t )^{r-1}{\rm d}x\\
& = \int_{ \{f\le\min B^f \}} \bigl(t-f(x) \bigr)^{r-1}{\rm d}x-\int_{ \{f\ge\max B^f \}} (f(x)-t )^{r-1}{\rm d}x
\end{split}
\end{equation} 
is  increasing in $t$. Thus $\phi_r^r$ is strictly convex, and $\tau_r^f$ is unique.

To show that $r\mapsto \tau_r^f$ is continuous on $ ]1,r_0 ]$, pick
any $1<r\le r_0$ and any sequence $(r_n)$ in $ ]1,r_0 ]$ with
$\lim_{n\to\infty}r_n=r$. Given $\varepsilon>0$, by the strict
convexity of $\phi_r$, there exists $\delta>0$ such that 
$\phi_r (\tau_r^f\pm\varepsilon )>\phi_r (\tau_r^f )+3\delta$. 
On the other hand, $\lim_{n\to\infty}\phi_{r_n}(t)=\phi_r(t)$ for
every $t\in\mathbb{R}$, by the Dominated Convergence Theorem. 
Hence for all sufficiently large $n$, 
$$
\phi_{r_n} (\tau_r^f\pm\varepsilon )>\phi_r (\tau_r^f )+2\delta\ \
\text{and}\ \ \phi_{r_n} (\tau_r^f )<\phi_r (\tau_r^f )+\delta \, ,
$$ 
from which it is clear that $ |\tau_{r_n}^f-\tau_r^f |<\varepsilon$. 
Since $\varepsilon>0$ was arbitrary, $r\mapsto \tau_r^f$ is continuous.
\end{proof}

For monotone functions, Theorem~\ref{le2} takes a particularly simple
form.

\begin{cor}\label{co2}
Assume that $f:I\to\overline{\mathbb{R}}$ is non-decreasing, and  
$f\in L^{r_0}(I)$ for some $r_0\ge1$. Then for every $1\le r\le r_0$, 
there exists $\tau_r^f\in\mathbb{R}$ such that  
$$
\left\|f-\tau_r^f\right\|_r\le \|f-t \|_r,\ \ \forall\
t\in\mathbb{R}\, . 
$$ 
Moreover, the following hold:
\begin{enumerate}
\item $\tau^f_r\in [f(\min I+),f(\max I-) ]$.
\item $ \|f-t \|_1= \|f-\tau_1^f \|_1$ if and only if $t\in Q_{\frac{1}{2}(\min I+\max I)}^{f^{-1}}$.
\item For $1<r\le r_0$, the number $\tau_r^f$ is unique and $r\mapsto \tau_r^f$ is continuous.
\end{enumerate}
\end{cor}

\begin{rem}\label{re1a}
(i) If $f\in L^2(I)$ then simply
$\tau_2^f=\frac{1}{\lambda(I)}\int_If(x)\, {\rm d}x$.

(ii) For $r=1$, Corollary~\ref{co2} immediately yields
Lemma~\ref{le3}. Indeed, under the assumptions of the latter, $f^{-1}
|_{[a,b]}\in L^1 ([a,b] ) $, and $ \|f- (a{\bf 1}_{ [\min I,\xi
  [}+b{\bf 1}_{ [\xi,\max I ]} ) \|_1$ is minimal if and only if $
\|f^{-1} |_{[a,b]}-\xi \|_1 $ is minimal. By Corollary~\ref{co2}, this
is the case precisely if $\xi\in Q^{g^{-1}}_{\frac{1}{2}(a+b)}$ 
with $g=f^{-1} |_{[a,b]} $, which by Proposition~\ref{le1} is
equivalent to $\xi\in Q^f_{\frac{1}{2}(a+b)}$. 
\end{rem}

Given $r>1$, the number $\tau_r^f$ depends on $f$ in a monotone and continuous way, as the following two simple observations show.

\begin{prop}\label{mono}
Assume that $f,\ g\in L^r(I)$ for some $r>1$, and $f\le g$. Then $\tau_r^f\le\tau_r^g$, and $\tau_r^f=\tau_r^g$ if and only if $f=g$ a.e..
\end{prop}

\begin{lem}\label{le4}
Assume that $f,\ f_n\in L^{r_0}(I)$ for some $r_0>1$ and all $n\in\mathbb{N}$. If $\lim_{n\to\infty}f_n=f$ in $L^{r_0}(I)$, then $\lim_{n\to\infty}\tau_r^{f_n}=\tau_r^f$ locally uniformly on $]1,r_0]$.
\end{lem}

\begin{proof}
Since $f_n\to f$ in $L^{r_0}(I)$ and $I$ is bounded,
$\sup_{n\in\mathbb{N}} \|f_n \|_{r_0}<+\infty$ and, for all $r\in \, ]1,r_0 ]$ and $n\in\mathbb{N}$,
\[
\begin{split}
\left|\tau_r^{f_n}\right|=&\ \lambda(I)^{-1/r}\left\|\tau_r^{f_n}\right\|_r\le
\lambda(I)^{-1/r} (\left\|f_n-\tau^{f_n}_r\right\|_r+\|f_n\|_r )\\
\le&\ 2\lambda(I)^{-1/r} \|f_n \|_r\le2\lambda(I)^{-1/r_0} \|f_n \|_{r_0},
\end{split}
\]
by H\"{o}lder's inequality. This shows that $ (\tau_r^{f_n} )$ is
uniformly bounded on $]1,r_0]$.

Fix any $1<s<r_0$. To prove that
$\lim_{n\to \infty}\tau_r^{f_n}=\tau_r^f$ uniformly on
$ [s,r_0 ]$, suppose by way of contradiction that there exists
$\varepsilon_0>0$, a sequence $(r_j)$ in $ [s,r_0 ]$ and an
increasing sequence $(n_j)$ in $\mathbb{N}$ such
that 
$$
\left|\tau_{r_j}^f-\tau_{r_j}^{f_{n_j}}\right|\ge\varepsilon_0,\quad
\forall\ j\in\mathbb{N}\, .
$$ 
Assume w.o.l.g.\ that $r_j\to r^*$ and, by
the uniform boundedness of $ (\tau_r^{f_n} )$,
$\tau_{r_j}^{f_{n_j}}\to\tau^*\in\mathbb{R}$. Since $r\mapsto\tau_r^f$
is continuous at $r^*$, it follows that
\begin{equation}\label{4} 
|\tau_{r^*}^f-\tau^* |\ge\varepsilon_0\, .
\end{equation} 
On the other hand,
\[
\begin{split}
  \left\|f-\tau_{r_j}^{f_{n_j}}\right\|_{r_j}\le&\left\|f-f_{n_j}\right\|_{r_j}
  +\left\|f_{n_j}-\tau_{r_j}^{f_{n_j}}\right\|_{r_j}\le\left\|f-f_{n_j}\right
  \|_{r_j}
  +\left\|f_{n_j}-\tau_{r_j}^f\right\|_{r_j}\\
  \le&\ 2\left\|f-f_{n_j}\right\|_{r_j}+\left\|f-\tau_{r_j}^f\right\|_{r_j},
\end{split}
\] 
and letting $j\to\infty$ yields,
$\left\|f-\tau^*\right\|_{r^*}\le\left\|f-\tau_{r^*}^f\right\|_{r^*}$
since $(r,t)\mapsto \|f-t \|_r$ is continuous. 
By Theorem~\ref{le2}(iii), $\tau^*=\tau_{r^*}^f$, which clearly contradicts \eqref{4}.
\end{proof}

\begin{rem}
In Lemma~\ref{le4}, the convergence $\tau_r^{f_n}\to\tau_r^f$ in
general is not uniform on $ ]1,r_0 ]$. To see this, take for example
$I=[0,2]$ and $f_n=2\cdot{\bf 1}_{[1+2^{-n},2]}$ for all
$n\in\mathbb{N}$. With $f=2\cdot{\bf 1}_{[1,2]}$, clearly, 
$f, f_n\in L^{\infty}(I)$ and $\lim_{n\to\infty} f_n=f$ in $L^r(I)$ for
every $r\ge1$. Still, $\lim_{r\downarrow1}\tau_r^{f_n}=0$ for every
$n$, whereas $\tau_r^f=1$ for all $r>1$.
\end{rem}

Note that if $f:I\to\mathbb{R}$ is affine, i.e., $f(x)=ax+b$ for all
$x\in I$ and the appropriate $a, b\in\mathbb{R}$, then
$\tau_r^f=f\bigl(\frac{1}{2} (\min I+\max I ) \bigr)$ for all $r>1$. In this context,
Lemma~\ref{le4} can be given a slightly stronger, quantitative form.

\begin{prop}\label{quan}
Assume that $f:I\to\mathbb{R}$ is measurable, and let
$\xi=\frac{1}{2}(\min I+\max I)$. If, for some $a, b, c\in\mathbb{R}$, 
$$
 |f(x)-(ax+b) |\le c |x-\xi |,\quad \forall\ x\in I\, ,
$$
then $f\in L^{\infty}(I)$, and $\left|\tau_r^f-f(\xi)\right|\le\frac{1}{2}c\lambda(I)$ for every $r>1$.
\end{prop}

The remainder of this section studies how, given $f$, the number
$\tau_r^f$ depends on $r$. First, this dependence is illustrated by an
example, where for simplicity $f\in L^{\infty}(I)$ is a non-decreasing step function.

\begin{example}\label{ex1}
Let $I=[0,8]$.

(i) Consider the function $f=(-4)\cdot{\bf 1}_{[0,1[}+4\cdot{\bf
  1}_{[5,8]}$, for which $B^f=\{0\}$, and clearly $0\le \tau_r^f \le
4$ for every $r>1$. By \eqref{phid}, 
\begin{equation}\label{000} 
(\tau_r^f+4 )^{r-1}+4 (\tau_r^f )^{r-1}=3 (4-
\tau_r^f )^{r-1}\, ,
\end{equation} 
and using \eqref{000}, it is readily deduced that
$\tau_{1+}^f:=\lim_{r\downarrow1}\tau_r^f=0$, 
but also $\tau_{\infty}^f:=\lim_{r\to+\infty}\tau_r^f=0$. On the other
hand, $\tau_2^f=1$, and hence $r\mapsto\tau_r^f$ is non-monotone; 
see Figure~\ref{f1}. Note that in order for $r\mapsto\tau_r^f$ to be non-monotone, a step function $f$ has to attain at least three different values.

(ii) Consider the function $f=(-a){\bf 1}_{[0,1[}+(-1) \cdot {\bf 1}_{[1,4[}+{\bf 1}_{[4,5[}+b
{\bf 1}_{[5,8]}$ with real parameters $a, b>1$. In this case,
$B^f=[-1,1]$, and \eqref{phid} yields, for every $r>1$, 
$$ 
(\tau_r^f+a )^{r-1}+3 (\tau_r^f+1 )^{r-1}= (1-\tau_r^f )^
{r-1}+3 (b-\tau_r^f )^{r-1}\, ,
$$ 
from which it is straightforward to deduce that $\tau_{1+}^f$ exists
and equals the unique real root of 
\begin{equation}\label{22}
g_{a,b}(\tau):= (3b+a+4 )\tau^3-3 (b^2+b-a-1 )\tau^2
+ (b^3+3b^2+3a+1 )\tau-b^3+a=0\, .
\end{equation} 
Given $\tau\in\, ]$$-1,1[$, note that
$\lim_{a\to+\infty}g_{a,b}(\tau)=+\infty$ for every $b>1$, and
$\lim_{b\to+\infty}g_{a,b}(\tau)=-\infty$  for every $a>1$.  
By the Intermediate Value Theorem, there exists
$\overline{a}=\overline{a}(\tau),\ \overline{b}=\overline{b}(\tau)$ 
such that $g_{\overline{a},\overline{b}}(\tau)=0$. Since the real root
of \eqref{22} is unique, $\tau^f_{1+}=\tau$. This shows that with
$a,b>1$ chosen appropriately, 
$\tau_{1+}^f$ can have any value in $]$$-1,1[$. Note that, similarly to (i), $\tau_{\infty}^f=\frac{1}{2}(b-a)$.
\end{example}

\begin{figure}[h] 
\psfrag{tv1}[r]{\small $1$}
\psfrag{tv05}[r]{\small $\frac12$}
\psfrag{tv0}[r]{\small $0$}
\psfrag{tv4}[r]{\small $4$}
\psfrag{tvm4}[r]{\small $-4$}
\psfrag{th1}[]{\small $1$}
\psfrag{th5}[]{\small $5$}
\psfrag{th8}[]{\small $8$}
\psfrag{ttau}[r]{$\tau_r^f$}
\psfrag{tr}[]{$r$}
\psfrag{tx}[]{$x$}
\psfrag{tfx}[r]{$f(x)$}
\psfrag{ti}[r]{$I = [0,8]$}
\begin{center}
\includegraphics{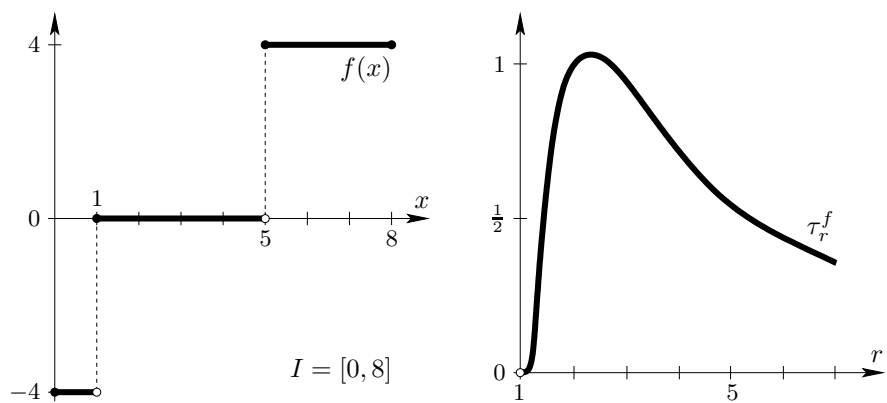}
\end{center}
\caption{For the (non-decreasing) function $f= (-4) \cdot {\bf 1}_{[0,1[}
  + 4 \cdot {\bf 1}_{[5,8]}$ the value of $\tau_r^f$ depends
  non-monotonically on $r$; see Example \ref{ex1}(i).}\label{f1}
\end{figure}

As seen in Example~\ref{ex1}, the number $\tau_r^f$ may depend on $r$
in a non-monotone way. In both cases considered, however, the limits
$\tau_{1+}^f=\lim_{r\downarrow1}\tau_r^f$ and
$\tau_{\infty}^f=\lim_{r\to+\infty}\tau_r^f$ exist. Also, by modifying
Example~\ref{ex1}(ii) appropriately, it is clear that, given any
compact interval $J\subset\mathbb{R}$ and any $\tau\in J$, one can
find $f\in L^{\infty}(I)$ with $B^f=J$ and $\tau_{1+}^f=\tau$. In
fact, one can choose $f$ to be a non-decreasing step function.

This section concludes with a demonstration that, just as in
Example~\ref{ex1}, $\tau_{1+}^f$ exists always (Theorem~\ref{thm1}),
whereas, unlike in Example~\ref{ex1}, $\tau_{\infty}^f$ may not 
exist (Example~\ref{ex3}).

\begin{theorem}\label{thm1} 
Assume that $f\in L^{r_0}(I)$ for some $r_0>1$. Then $\tau_{1+}^f$ exists, and $\tau_{1+}^f\in B^f$.
\end{theorem}

\begin{proof}
We first show that 
\begin{equation}\label{13}
  [{\liminf}_{r\downarrow1}\tau_r^f,{\limsup}_{r\downarrow1}\tau_r^f
  ]\subset B^f\, ,
\end{equation}
and then that $\lim_{r\downarrow1}\tau_r^f$ exists.
For any $1<r\le r_0$, let $\phi_r$ be defined as in \eqref{phi}.
Recall that $\phi_r$ is convex, and $r\mapsto\phi_r(t)$ is continuous
and non-decreasing for any $t\in\mathbb{R}$. 
Assume that $r_n\downarrow1$ with $\tau_{r_n}^f\to\tau$. Then $\phi_1\left(\tau_{r_n}^f\right)\le\phi_{r_n}\left(\tau_{r_n}^f\right)\le
\phi_{r_n}(t)$, and hence $\phi_1(\tau)=\lim_{n\to\infty}\phi_1\left(\tau_{r_n}^f\right)\le\lim_{n\to
\infty}\phi_{r_n}(t)=\phi_1(t)$. Since $t\in\mathbb{R}$ was arbitrary, Theorem~\ref{le2}(ii) yields $\tau\in B^f$, which in turn establishes \eqref{13}.

It remains to show that $\lim_{r\downarrow1}\tau_r^f$ exists, which is
non-trivial only if $B^f$ is non-degenerate. In this case, define
$\Psi:\overset{\circ}{B^f}\to\mathbb{R}$ as  
$$
\Psi(t)=\int_{ \{f\le\min B^f \}}\log\bigl(t-f(x) \bigr) \, {\rm d}x-\int_{
  \{f\ge\max B^f \}}\log (f(x)-t )\, {\rm d}x,\quad \forall t\in
\overset{\circ}{B^f}\, .
$$
Note that $\Psi$ is well-defined and continuous. Moreover, if $t,
u\in B^f$ with $t<u$ then, as $B^f\neq I$, 
\[
  \Psi(t)-\Psi(u)=\int_{ \{f\le \min B^f
    \}}\log\frac{t-f(x)}{u-f(x)}\, {\rm d}x+\int_{ \{f\ge\max B^f
    \}}\log\frac{f(x)-u}{f(x)-t}\, {\rm d}x<0,
\] 
showing that $\Psi$ is increasing.
By \eqref{phid}, $t\mapsto\frac{\lambda(I)}{r}\frac{{\rm d}\phi^r_r(t)}{{\rm d}t}$ is a real-valued increasing function.
To compare the latter to $\Psi$, notice the elementary inequality
\begin{equation}\label{23}
   |y^{r-1}-1-(r-1)\log y |\le(r-1)^2e^{|\log y|},\quad \forall y>0, 1\le r\le2.
\end{equation}
With Lemma~\ref{le5} and \eqref{23}, for any fixed $0<\varepsilon<\min
\{1,\frac{1}{2}\lambda (B^f ) \}$, there exists $C_{\varepsilon}>0$ such that
\begin{equation}\label{16}
 \left| \frac{\lambda(I)}{r}\frac{{\rm d}\phi_r^r(t)}{{\rm d}t}-(r-1)\Psi(t)
 \right|
\le C_{\varepsilon}(r-1)^2,\quad \forall  1<r\le 2,\ t\in [\min B^f+\varepsilon,\max B^f-\varepsilon ].
\end{equation}
Since $\Psi$ is increasing, three cases may be distinguished:

\noindent
(i) $\Psi(\tau)=0$ for a unique $\tau\in\overset{\circ}{B^f}$. Pick
$\varepsilon>0$ so that $\min B^f+\varepsilon<\tau<\max
B^f-\varepsilon$. Then for every $\delta>0$, \eqref{16} implies
$\frac{{\rm d}\phi_r}{{\rm d}t} (\tau+\delta )>0$ and $\frac{{\rm
    d}\phi_r}{{\rm d}t} (\tau-\delta )<0$ for all $r>1$ sufficiently small.
It follows that $\tau_r^f\in[\tau-\delta,\tau+\delta]$ for all $r>1$
sufficiently small, and since $\delta>0$ was arbitrary, $\lim_{r\downarrow1}\tau_r^f=\tau$.

\noindent
(ii) $\Psi(\tau)>0$ for all $\tau\in\overset{\circ}{B^f}$. Similarly
to case (i), for every $\delta>0$, \eqref{16} yields $\frac{{\rm
    d}\phi_r}{{\rm d}t} (\min B^f+\delta )>0$ for all $r>1$
sufficiently small. This implies that $\tau_r^f<\min B^f+\delta$  for
all $r>1$ sufficiently small and hence
$\limsup_{r\downarrow1}\tau_r^f\le\min B^f$. 
By \eqref{13}, $\lim_{r\downarrow1}\tau_r^f=\min B^f$.

\noindent
(iii) $\Psi(\tau)<0$ for all $\tau\in\overset{\circ}{B^f}$. This case is completely analogous to (ii), with $\lim_{r\downarrow1}\tau_r^f=\max B^f$.
\end{proof}

\begin{cor}\label{co3}
Assume that $f:I\to\overline{\mathbb{R}}$ is non-decreasing, and $f\in L^{r_0}(I)$ for some $r_0>1$. Then $\tau_{1+}^f$ exists, and $\tau_{1+}^f\in Q_{\frac{1}{2}(\min I+\max I)}^{f^{-1}}$.
\end{cor}

Recall that in Example~\ref{ex1} the limit $\tau_{\infty}^f$ also
exists. This is a consequence of the fact that $f$ is bounded,
together with the following simple observation.

\begin{theorem}\label{th7}
Assume that $f\in \bigcap_{r\ge1} L^r(I)$. If $f^-\in L^{\infty}(I)$ or $f^+\in L^{\infty}(I)$, then $\lim_{r\to+\infty}\tau_r^f=\frac{1}{2} ({\rm essinf}_I f+{\rm esssup}_I f )$.
\end{theorem}

\begin{proof}
Let $f$ be non-constant (otherwise, $r\mapsto\tau_r^f$ is constant,
too), and assume that $f^-\in L^{\infty}(I)$, that is, ${\rm
  essinf}_If>-\infty$. (The case $f^+\in L^{\infty}(I)$ is completely
analogous.) Let $u=\frac{1}{2} ({\rm essinf}_I f+{\rm esssup}_I f )$
for convenience, fix any ${\rm essinf}_I f<t<u$, and let
$\delta=t-{\rm essinf}_I f$. For $\tau<t$, note that $\tau-{\rm
  essinf}_I f<\delta$ and $\lambda ( \{f\ge\tau+\delta \} )>0$, 
and hence, with \eqref{phid},
\begin{align*}
\frac{\lambda(I)}{r\delta^{r-1}}\frac{{\rm d}\phi_r^r(\tau)}{{\rm
    d}\tau} &= \int_{ \{f<\tau \}}\left(\frac{\tau-f(x)}{\delta}\right)^{r-1}{
\rm d}x-\int_{ \{f>\tau \}}\left(\frac{f(x)-\tau}{\delta}\right)^{r-1}{\rm d}x\\
& \le \ \lambda(I)\left(\frac{\tau-{\rm essinf}_I f}{\delta}\right)^{r-1}-\int_{ \{f\ge\tau+\delta \}}\left(\frac{f(x)-\tau}
{\delta}\right)^{r-1}{\rm d}x\\
& \quad -\int_{ \{\tau\le f<\tau+\delta \}}\left(\frac{f(x)-\tau}{\delta}\right)^{r-1}{\rm d}x\\
&\le  \lambda(I)\left(\frac{\tau-{\rm essinf}_I f}{\delta}\right)^{r-1}-\lambda ( \{f\ge\tau+\delta \} )<0,
\end{align*}
for all sufficiently large $r$. Thus
$\liminf_{r\to+\infty}\tau_r^f\ge\tau$, and since $t$ and $\tau<t$
were arbitrary, $\liminf_{r\to+\infty}\tau_r^f\ge u$. 
A similar argument shows $\limsup_{r\to+\infty}\tau_r^f\le u$.
\end{proof}

\begin{cor}\label{co4}
If $f\in \bigcap_{r\ge 1}L^r(I)$ is non-decreasing and either
$f(\min I+)>-\infty$ or $f(\max I-)$ $<+\infty$, then
$\lim_{r\to+\infty}\tau_r^f=\frac{1}{2} \bigl( f(\min I+)+f(\max I-) \bigr)$.
\end{cor}
The final example shows that, unlike in Example~\ref{ex1}, $\lim_{r\to+\infty}\tau_r^f$ may not exist if $f$ is unbounded.

\begin{example}\label{ex3}
\noindent
Consider the function $f:I\to\mathbb{R}$ given by
$$
f=\sum\nolimits_{n=0}^{\infty}2n(-1)^{n-1}{\bf 1}_{I_n}\, ,
$$
where $I:=\overline{ \bigcup_{n=0}^{\infty} I_n}$,
and $I_0,\ I_1,\ \cdots$ are pairwise disjoint, contiguous half-open
intervals, with $I_1$ to the right of $I_0$, and generally $I_{2n+1}$
immediately to the right of $I_{2n-1}$, as well as $I_{2n+2}$
immediately to the left of $I_{2n}$. (Clearly, $f$ is non-decreasing
on $\overset{\circ}{I}$.) The lengths $\lambda_n:=\lambda(I_n)>0$ will
be determined by induction shortly, subject to the requirement that
$\lambda_{n+1}\le\frac{1}{2}\lambda_n$ for all $n\ge0$.
Thus $I$ is a non-degenerate, closed interval of length
$\sum_{n\ge0}\lambda_n\le2\lambda_0$, and
$f\in\bigcap_{r\ge 1}L^r(I)$ but clearly $f\notin
L^{\infty}(I)$. For each $N\in\mathbb{N}$, let
$f_N=\sum_{n=0}^N2n(-1)^{n-1}{\bf 1}_{I_n}$ and note that
$\lim_{r\to+\infty}\tau_r^{f_N}=(-1)^{N-1}$, by Theorem~\ref{th7}. 
Moreover, 
\begin{equation}\label{0-1} \|f_{N+1}-f_N
  \|_r=2(N+1)\lambda_{N+1}^{1/r},\quad \forall  r>1,\ N\in\mathbb{N}\, .
\end{equation} 
Let $\lambda_0=1,\ r_0=1$, and assume that $\lambda_1,\
\lambda_2,\cdots,\lambda_N$ with
$\lambda_n\le\frac{1}{2}\lambda_{n-1}$ as well as $r_0<r_1<\cdots<r_N$
with $r_n\ge\max\{r_{n-1},n+1\}$ for $n=1,\cdots,N$ have been chosen
in such a way that, for every $1\le n\le N$, 
\begin{equation}\label{0-2} |\tau_{r_j}^{f_n}-(-1)^{j-1}
  |<2^{1-j}-2^{-n},\quad \forall 1\le j\le n\, .
\end{equation} 
For $N=1$, clearly such a choice is possible. By Lemma~\ref{le4} and
\eqref{0-1}, choosing $\lambda_{N+1}\le\frac{1}{2}\lambda_N$
sufficiently small guarantees that 
$$ 
|\tau_r^{f_{N+1}}-\tau_r^{f_N}
|<2^{-(N+1)},\quad \forall r\in [r_1,r_N ]\, ,
$$ 
and consequently
\begin{align*}
   |\tau_{r_j}^{f_{N+1}}-(-1)^{j-1} | & \le |\tau_{r_j}^{f_{N+1}}-\tau_{r_
  j}^{f_N} |+ |\tau_{r_j}^{f_N}-(-1)^{j-1} |\\
  & <  2^{-(N+1)}+2^{1-j}-2^{-N}=2^{1-j}-2^{-(N+1)},\quad \forall  j=1,\cdots,N.
\end{align*}
Also, choose $r_{N+1}\ge\max \{r_N,N+2 \}$ such
that $|\tau_{r_{N+1}}^{f_{N+1}}-(-1)^N |<2^{-(N+1)}$. Thus
\eqref{0-2} holds for $n=N+1$, and in fact for all $n\in\mathbb{N}$,
by induction. Furthermore, note that, given any $r>1$,
\[ 
\|f_N-f \|_r=\left(\sum\nolimits_{n>N}(2n)^r\lambda_n\right)^{1/r}\le2
\lambda_0^{1/r}\left(\sum\nolimits_{n>N}n^r2^{-n}\right)^{1/r} \!\!\! \to 0 \quad
\text{as}\ N\to\infty\, ,
\]
and so in particular $\lim_{N\to\infty} \|f_N-f \|_{r_j}=0$ for every
$j\in\mathbb{N}$. By Lemma~\ref{le4}, $ |\tau_{r_j}^{f_N}-\tau_{r_j}^f
|<2^{-j}$ for all sufficiently large $N$, which, 
together with \eqref{0-2}, yields $|\tau_{r_j}^f-(-1)^{j-1}
|<3\cdot2^{-j}$. Since $j\in\mathbb{N}$ was arbitrary and
$r_j\uparrow+\infty$, this shows that
$\liminf_{r\to+\infty}\tau_r^f\le-1$ and
$\limsup_{r\to+\infty}\tau_r^f\ge1$. On the other hand, using
\eqref{phid}, it is readily confirmed that $t\frac{{\rm d}}{{\rm d}t}
\|f-t \|_r^r>0$ for $t=\pm1$ and all $r>1$, and consequently $
|\tau_r^f |<1$. Thus $\liminf_{r\to+\infty}\tau_r^f=-1$ and $\limsup_{r\to+\infty}\tau_r^f=1$.
\end{example}

By modifying Example~\ref{ex3} appropriately, it is straightforward to establish
\begin{prop}
Given any (bounded) interval $I\subset\mathbb{R}$ and numbers
$-\infty\le a\le b\le+\infty$, there exists a non-decreasing function
$f\in \bigcap_{r\ge 1} L^r(I)$ such that
$\liminf_{r\to+\infty}\tau_r^f=a$ and $\limsup_{r\to+\infty}\tau_r^f=b$.
\end{prop}

\section{Best constrained approximations}\label{sec5}

In this section, we apply results established in previous sections,
notably Lemma~\ref{le3} and Theorem~\ref{le2}, 
to investigate best constrained approximations of
$\mu\in\mathcal{P}_r$, i.e., approximations of $\mu$ by finitely
supported probabilities for which either locations
(Subsection~\ref{subloc}) or weights (Subsection~\ref{subwei}) are
prescribed. We establish existence of best constrained approximations 
and study their behaviour as the number of atoms goes to infinity. 
Finally, in Subsection~\ref{subbest} we relate these results to the
classical theory of best (unconstrained) approximations. The main
results of this section are Theorems \ref{th2}, \ref{th1}, \ref{th6}, \ref{low},
\ref{th5-1}, and \ref{slow}.

First, we fix a few notations specific to this section.
Given $n\in\mathbb{N}$, let $\Xi_n= \{{\bf x}\in\mathbb{R}^n: x_{ 1}\le\cdots\le x_{ n} \}$ and
$\Pi_n= \{{\bf p}\in\mathbb{R}^n: p_{ i}\ge0,\
\sum_{i=1}^np_{ i}=1 \}$. For any ${\bf x}\in\Xi_n$, the conventions
$x_{ 0}=-\infty$ and $x_{ n+1}=+\infty$ are adopted, and for any
${\bf p}\in\Pi_n$, let $P_{ i}=\sum_{j=1}^ip_{ j}$, $i=0,1,\cdots,n$; 
note that $P_{ 0}=0$ and $P_{ n}=1$. Given ${\bf x}\in\Xi_n$ and ${\bf
  p}\in\Pi_n$, let $\delta_{{\bf x}}^{{\bf p}} =\sum_{i=1}^np_{ i}\delta_{x_{ i}}$.
Throughout, usage of the symbol $\delta_{{\bf x}}^{{\bf p}} $
tacitly assumes that ${\bf x}\in\Xi_n,\ {\bf p}\in\Pi_n$, with $n\in\mathbb{N}$
either specified explicitly or else clear from the context.

\subsection{Best approximations with prescribed locations}\label{subloc}

Let $\mu\in\mathcal{P}_r$ for some $r\ge1$, and $n\in\mathbb{N}$. 
Given ${\bf x}\in \Xi_n$, call $\delta_{{\bf x}}^{{\bf p}} $ with
${\bf p}\in\Pi_n$ a {\em best $r$-approximation of $\mu$, given ${\bf x}$} if 
$$
d_r (\delta_{{\bf x}}^{{\bf p}} ,\mu )\le d_r (\delta_{{\bf x}}^{{\bf q}} ,\mu ),\quad \forall 
{\bf q}\in\Pi_n\, .
$$
Denote by $\delta_{{\bf x}}^{\bullet}$ any (possibly not unique) best
$r$-approximation of $\mu$, given ${\bf x}$. (Note that $\delta_{{\bf x}}^{\bullet}$
also depends on $r$. In the interest of readability, this dependence
is made explicit by a subscript only when necessary to avoid ambiguities.)

The existence of best $r$-approximations with prescribed locations can
be established using the results of Sections~\ref{Mono} and \ref{Lr}. 
\begin{theorem} \label{th2}
Assume that $\mu\in\mathcal{P}_r$ for some $r\ge1$, and
$n\in\mathbb{N}$. For every ${\bf x}\in \Xi_n$, there exists a best
$r$-approximation of $\mu$, given ${\bf x}$. Moreover,
$d_r\left(\delta_{{\bf x}}^{{\bf p}},\mu\right)=d_r\left(\delta_{{\bf x}}^{\bullet},\mu\right)$
with ${\bf p}\in\Pi_n$ if and only if, for every $i=1,\cdots,n$,
\begin{equation}\label{cl}
x_{ i}<x_{ i+1}\ \text{implies}\ P_{ i}\in
Q^{F^{-1}_{\mu}}_{\frac{1}{2} (x_{ i}+x_{ i+1} )}\, .
\end{equation}
\end{theorem}

\begin{proof}
For convenience, let
$A_i=Q^{F^{-1}_{\mu}}_{\frac{1}{2} (x_{ i}+x_{ i+1} )}$ for $0\le
i\le n$; note that $A_0=[-\infty,0],\ A_n=[1,+\infty]$, and every
$A_i$ is a compact (possibly one-point) interval, by
Proposition~\ref{le1}. Since the theorem trivially is correct for
$n=1$, henceforth assume $n\ge2$. We first establish \eqref{cl}, 
as the asserted existence of best $r$-approximations will follow directly from it.

Labelling ${\bf x}\in\Xi_n$ as 
\begin{equation}\label{x}
x_{ i_0+1}=\cdots=x_{ i_1}<x_{ i_1+1}=\cdots=x_{ i_2}<
x_{ i_2+1}=\cdots<\cdots<x_{ i_{m-1}+1}=\cdots=x_{ i_m}
\end{equation} 
with integers $j\le i_j\le n$ for $1\le j\le m\le n$, and $i_0=0$,
$i_m=n$, note first that $d_r(\delta_{{\bf x}}^{{\bf
    p}},\mu)=d_r(\delta_{\overline{{\bf x}}}^
{\overline{{\bf p}}},\mu )$, where $\overline{{\bf x}}\in\Xi_m$ and
$\overline{{\bf p}}\in\Pi_m$, with $\overline{x}_{ j}=x_{ i_j}$, and
$\overline{P}_{ j}=P_{ i_j}$ for $1\le j\le m$. 
Moreover, \eqref{cl} reduces to $\overline{P}_{ j}\in Q_{\frac{1}{2}(\overline{x}_{ j}+\overline{x}_{ j+1})}^{F_{\mu}^{-1}}$ for all $1\le j\le m-1$.
To establish \eqref{cl}, therefore, it can be assumed w.o.l.g. that $x_{ i}<x_{ i+1}$ for all $i$.

To prove that \eqref{cl} is necessary, let $\delta_{{\bf x}}^{{\bf p}}$ be a best
$r$-approximation of $\mu$, given $x$. Given any $1\le i\le n-1$, let
$\widetilde{{\bf p}}\in\Pi_n$ satisfy $\widetilde{p}_{ j}=p_{ j}$ for all
$j\neq i,i+1$, and $0\le\widetilde{p}_{ i}\le p_{ i}+p_{ i+1}$. Note
that $P_{ i-1}\le \widetilde{P}_{ i}\le P_{ i+1}$.

If $P_{ i-1}<P_{ i+1}$, then $d_r\left(\delta_{{\bf x}}^{{\bf p}},\mu\right)\le
d_r\left(\delta_{{\bf x}}^{\widetilde{{\bf p}}},\mu\right)$ implies
$$
 \|f_i- (x_{ i} {\bf 1}_{[P_{ i-1},P_{ i}[}+x_{ i+1} {\bf 1}
_{[P_{ i},P_{ i+1}]} ) \|_r\le
 \|f_i- (x_{ i}  {\bf 1}_{[P_{ i-1},\widetilde{P}_{ i}[}+x_{ i+1}
{\bf 1}_{[\widetilde{P}_{ i},P_{ i+1}]} ) \|_r \, ,
$$
with $f_i=F_{\mu}^{-1} |_{[P_{ i-1},P_{ i+1}]}$. Since
$\widetilde{P}_{ i}\in [P_{ i-1},P_{ i+1} ]$ was arbitrary,
Lemma~\ref{le3} and Proposition~\ref{le1} yield $P_{ i}\in Q_{\frac{1}{2} (x_{ i}+x_{ i+1} )}^{f_i}=A_i$.

If $P_{ i-1}=P_{ i+1}$, let $i^-$ and $i^+$ be the minimum and
maximum, respectively, of the (non-empty) set $ \{0\le j\le n:\
P_{ j}=P_{ i} \}$. Clearly, $0\le i^-\le i-1$, $i+1\le i^+\le n$, and
$i^+-i^-\ge2$. Assume first that $i^-=0$, in which case $i^+\le n-1$
and $P_{ i}=P_{ i^+}=0$. Lemma~\ref{le3}, applied to $f_{i^+}$ yields
$0\in A_{i^+}$. Recall that $A_i\subset\mathbb{I}$ and $\max
A_i\le\min A_{i^+}$, by Proposition~\ref{le1}. Thus $0\le\min
A_i\le\min A_{i^+}\le0$, and hence $0=P_{ i}\in A_i$. By a completely
analogous argument, the case of $i^+=n$, where $i^-\ge1$ and
$P_{ i}=P_{ i^-}=1$, leads to $1=P_{ i}\in A_i$. Finally, assume that
$1\le i^-<i^+\le n-1$. In this case, Lemma~\ref{le3}, applied to
$f_{i^-}$ and $f_{i^+}$ yields $P_{ i^-}\in A_{i^-}$ and $P_{ i^+}\in
A_{i^+}$, respectively. Thus $P_{ i}=P_{ i^-}=P_{ i^+}\in A_{i^-}\cap
A_{i^+}$. Since $j\mapsto\frac{1}{2} (x_{ j}+x_{ j+1} )$ is
increasing, Proposition~\ref{le1} implies that $A_i= \{P_{ i} \}$, and
hence trivially  $P_{ i}\in A_i$. This completes the proof that
\eqref{cl} holds whenever $d_r (\delta_{{\bf x}}^{{\bf p}},\mu )$ is minimal, i.e.,
\eqref{cl} is necessary.

To see that \eqref{cl} also is sufficient, let ${\bf p}\in\Pi_n$ satisfy
\eqref{cl} and consider $\widetilde{{\bf p}}\in\Pi_n$ with
$\widetilde{P}_i=\max A_i$ for all $i$. 
As $d_r (\delta_{{\bf x}}^{{\bf p}}, \mu)$ then has the same value for
every ${\bf p}$
satisfying (\ref{cl}), clearly, it is enough to show
that $d_r (\delta_{{\bf x}}^{{\bf p}},\mu )=d_r (\delta_{{\bf
    x}}^{\widetilde{{\bf p}}},\mu )$. 
To see the latter, note that by Proposition~\ref{le1}(i),
$P_{ i}\le\widetilde{P}_{ i}\le P_{ i+1}$, 
for all $1\le i\le n-1$, and  $ |x_{ i}-F_{\mu}^{-1}(t) |=
|x_{ i+1}-F_{\mu}^{-1}(t) |$ for all $P_{ i}<t<\widetilde{P}_{ i}$. 
Consequently,
\begin{align*}
  d_r\left(\delta_{{\bf x}}^{{\bf p}},\mu\right)^r & = \sum\nolimits_{i=1}^n\int_{P_{ i-1}}^{P_{ i}}\left
  |x_{ i}-F_{\mu}^{-1}(t)\right|^r{\rm d}t\\
  & = \sum\nolimits_{i=1}^n \left(\int_{P_{ i-1}}^{\widetilde{P}_{ i-1}}\left|x_{ i}-F_{\mu}
  ^{-1}(t)\right|^r{\rm d}t+\int^{P_{ i}}_{\widetilde{P}_{ i-1}}\left|x_{ i}-F_{\mu}^{-1}(t)\right|
  ^r{\rm d}t \right)\\
 & = \sum\nolimits_{i=1}^n \left(\int_{P_{ i-1}}^{\widetilde{P}_{ i-1}}\left|x_{ i-1}-F_{\mu}
^{-1}(t)\right|^r{\rm d}t+\int^{P_{ i}}_{\widetilde{P}_{ i-1}}\left|x_{ i}-F_{\mu}^{-1}(t)\right|^r
{\rm d}t \right) \\
  & = \sum\nolimits_{i=1}^n \left( \int_{P_{ i}}^{\widetilde{P}_{ i}}\left|x_{ i}-F_{\mu}^
  {-1}(t)\right|^r{\rm d}t+\int^{P_{ i}}_{\widetilde{P}_{ i-1}}\left|x_{ i}-F_{\mu}^{-1}(t)\right|^
  r{\rm d}t \right)\\
  & = \sum\nolimits_{i=1}^n\int^{\widetilde{P}_{ i}}_{\widetilde{P}_{ i-1}}\left|x_{ i}-
  F_{\mu}^{-1}(t)\right|^r{\rm d}t=d_r(\delta_{{\bf
    x}}^{\widetilde{{\bf p}}},\mu)^r.
\end{align*}
As indicated earlier, the asserted existence of a best
$r$-approximation of $\mu$, given $x$, is a direct consequence of
\eqref{cl}. Indeed, when $x\in\Xi_n$ is written as in \eqref{x},
Proposition~\ref{le1}(i) guarantees that the $m$ intervals
$A_{i_1-1},\ A_{i_2-1},\cdots,\ A_{i_m-1}\subset\mathbb{I}$ 
are arranged in such a way that $t\le u$ for all $t\in A_{i_j-1}$ and
$u\in A_{i_{j+1}-1}$. It is possible, therefore, to choose ${\bf p}\in\Pi_n$ satisfying \eqref{cl}.
\end{proof}

Given $\mu\in\mathcal{P}_r$ and ${\bf x}_n\in\Xi_n$ for all $n$, it is natural
to ask whether $d_r (\delta_{{\bf x}_n}^{\bullet},\mu )\to0$ as
$n\to\infty$. The following example illustrates that this may or may
not be the case.

\begin{example}\label{expe}
Let $\mu$ be the standard {\em exponential distribution} with
$F_{\mu}(x)=1-e^{-x}$ for all $x\ge0$. Note that
$\mu\in \bigcap_{r\ge 1} \mathcal{P}_r$. Given
${\bf x}_n=(1,2,\cdots,n)/\sqrt{n}\in\Xi_n$, Theorem~\ref{th2} yields a
unique best $r$-approximation of $\mu$, namely, $\delta_{{\bf
    x}_n}^{{\bf p}_n}$
with $P_{n,i}=F_{\mu} (\frac{2i+1}{2\sqrt{n}}
)=1-e^{-(2i+1)/(2\sqrt{n})}\ \text{for}\ 1\le i\le n-1$. It is readily
confirmed that $\lim_{n\to\infty}\sqrt{n}d_r (\delta_{{\bf x}_n}^{\bullet},\mu )=\frac{1}{2}
(r+1)^{-1/r}$ for every $r\ge1;$ in particular, therefore,
$\lim_{n\to\infty}d_r (\delta_{{\bf x}_n}^{{\bf p}_n},\mu )=0$. By contrast,
consider ${\bf y}_n=(0,2,\cdots,2n-2)\in\Xi_n$, for which
$\lim_{n\to\infty}d_r (\delta_{{\bf y}_n}^{\bullet},\mu )=d_r (\nu,\mu )>0$
with $\nu= (1-e^{-1} ) \delta_0+2\sinh1\sum_{i=1}^{\infty}e^{-2i}\delta_{2i}$. Note that
while every point in ${\rm supp}\ \mu=[0,+\infty]$ is the limit of an
appropriate sequence $ (x_{n,i_n} )$, this clearly is not the case for
$({\bf y}_n)$.
\end{example}

As Example~\ref{expe} suggests, a condition has to be imposed on
$({\bf x}_n)$, with ${\bf x}_n\in\Xi_n$ for all $n$, in order to
guarantee that $\lim_{n\to\infty}d_r (\delta_{{\bf x}_n}^{\bullet},\mu )=0$.

\begin{theorem}\label{thlo}
Assume that $\mu\in\mathcal{P}_r$ for some $r\ge1$, and ${\bf x}_n\in \Xi_n$
for every $n\in\mathbb{N}$. Then
$\lim_{n\to\infty}d_r (\delta_{{\bf x}_n}^{\bullet},\mu )=0$ if and only
if
\begin{equation}\label{loc}
\lim\nolimits_{n\to\infty} \min\nolimits_{1\le i\le
  n}|x-x_{n,i} |=0,\quad \forall  x\in \mathbb{R} \cap {\rm
  supp}\ \mu .
\end{equation}
In particular, \eqref{loc} holds whenever 
$$
\lim\nolimits_{n\to\infty} \bigl(F_{\mu} (x_{n,1} )+\max\nolimits_{1\le i\le
  n-1} (x_{n,i+1}-x_{n,i} )+1-F_{\mu} (x_{n,n} ) \bigr) =0\, .
$$
\end{theorem}

\begin{proof}
For convenience, let $P_{n,i}=F_{\mu} \bigl(\frac{1}{2}
(x_{n,i}+x_{n,i+1} ) \bigr)$ for all $n\in\mathbb{N}$ and $0\le i\le n$, as well as $A=\mathbb{I}\setminus \{P_{n,i}:\ n\in\mathbb{N},\ 0\le i\le n \}$
and $f_n=F_{\delta_{{\bf x}_n}^{{\bf p}_n}}$. Note that  $
|F_{\mu}^{-1}(t)-x_{n,i} |=\min\nolimits_{1\le j\le n}  |F_{\mu}^{-1}(t)-x_{n,j} |$ whenever $P_{n,i-1}<t<P_{n,i}$, and hence
\begin{equation}\label{de}  
|F_{\mu}^{-1}(t)-f_n^{-1}(t) |=\min\nolimits_{1\le j\le n}
|F_{\mu}^{-1}(t)-x_{n,j} |,\quad \forall  t\in A.
\end{equation}

We first show that \eqref{loc} is necessary. To see this, assume that
\eqref{loc} fails. Then, with the appropriate $\varepsilon>0,\
x\in{\rm supp}\ \mu$ and sequence $(n_k)$, 
$$
\min\nolimits_{1\le i\le n_k} |x-x_{n_k,i} |\ge2\varepsilon,\quad
\forall 
k\in\mathbb{N}.
$$
Since $f_n$ is constant on $ [x-\varepsilon,x+\varepsilon ]$ whereas
$F_{\mu}$ is not,  
$$
d_1 (\delta_{{\bf x}_{n_k}}^{\bullet},\mu )=d_1 (\delta_{{\bf
    x}_{n_k}}^{{\bf p}_{n_k}},
\mu )\ge\min\nolimits_{c\in\mathbb{R}}\int_{x-\varepsilon}^{x+\varepsilon} |F_{\mu}
(y)-c |\, {\rm d}y>0,\quad \forall k\in\mathbb{N},
$$
and so $\limsup_{n\to\infty}d_r (\delta_{{\bf x}_n}^{\bullet},\mu )>0$ as well.

To see that \eqref{loc} also is sufficient, note first that if
$F_{\mu}^{-1}$ is continuous at $t\in A$, then $F_{\mu}^{-1}(t)\in{\rm
  supp}\ \mu$, and hence $f_n^{-1}(t)\to F_{\mu}^{-1}(t)$, by
\eqref{de}. Since $F_{\mu}^{-1}$ is monotone, $f_n^{-1}\to
F_{\mu}^{-1}$ a.e. on $\mathbb{I}$. If ${\rm supp}\ \mu$ is bounded
then $f_n^{-1}\to F_{\mu}^{-1}$ in $L^r(\mathbb{I})$, by the Dominated
Convergence Theorem, i.e., $\lim_{n\to\infty}
d_r\left(\delta_{{\bf x}_n}^{{\bf p}_n},\mu\right)=0$, and thus $\lim_{n\to\infty}
d_r\left(\delta_{{\bf x}_n}^{\bullet},\mu\right)=0$. If, on the other hand,
${\rm supp}\ \mu$ is unbounded, then, given any $\varepsilon>0$,
choose $\nu\in\mathcal{P}$ with bounded support and $d_r (\mu,\nu
)<\varepsilon$. Then $d_r (\delta_{{\bf x}_n}^{\bullet},\mu )\le d_r
(\widetilde{\delta}_{{\bf x}_n}^{\bullet},\mu )\le d_r
(\widetilde{\delta}_{{\bf x}_n}^{\bullet},\nu )+d_r(\nu,\mu)$, where
$\widetilde{\delta}_{{\bf x}_n}^{\bullet}$ denotes a best $r$-approximation
of $\nu$, given ${\bf x}_n$. By the above, $\limsup_{n\to\infty}d_r
(\delta_{{\bf x}_n}^{\bullet},\mu )\le\varepsilon$, and since
$\varepsilon>0$ was arbitrary, 
$\lim_{n\to\infty}d_r (\delta_{{\bf x}_n}^{\bullet},\mu )=0$.
\end{proof}

\begin{example} \label{sq}
Let $\mu$ be the ${\rm Beta}(2,1)$ distribution, i.e.,
$F_{\mu}(x)=x^2$ for all $x\in\mathbb{I}$, and consider
${\bf x}_n=(1,\sqrt{2},\cdots,\sqrt{n})/\sqrt{n}\in\Xi_n$. By
Theorem~\ref{thlo}, $\lim_{n\to\infty}d_r
(\delta_{{\bf x}_n}^{\bullet},\mu )=0$ for every $r\ge1$. Unlike in
Example~\ref{expe}, however, the rate of convergence depends on
$r$: With $\alpha_r=\frac{1}{2}+\frac{1}{\max\{2,r\}}$ and the
appropriate $0<C_r<+\infty$, 
$$
\lim\nolimits_{n\to\infty}n^{\alpha_r}d_r
   (\delta_{{\bf x}_n}^{\bullet},\mu )=C_r
$$ 
whenever $r\neq2$, whereas 
$$
\lim\nolimits_{n\to\infty}\frac{n}{\sqrt{\log n}}\, d_2
(\delta_{{\bf x}_n}^{\bullet},\mu )=\frac{1}{4\sqrt{3}}\, .
$$
Thus $ \bigl(d_r (\delta_{{\bf x}_n}^{\bullet},\mu ) \bigr)$ decays like
$ (n^{-\alpha_r} )$ and $ ( n^{-1}\sqrt{\log n} )$ if $r\neq2$
and $r=2$, respectively.
\end{example}

\subsection{Best approximations with prescribed weights}\label{subwei}

Let $\mu\in\mathcal{P}_r$ for some $r\ge1$, and
$n\in\mathbb{N}$. Given ${\bf p}\in\Pi_n$, call $\delta_{{\bf
    x}}^{{\bf p}}$ with ${\bf x} \in
\Xi_n$ a {\em best $r$-approximation of $\mu$, given ${\bf p}$}  if 
$$
d_r
(\delta_{{\bf x}}^{{\bf p}},\mu )\le d_r (\delta_{{\bf y}}^{{\bf
    p}},\mu ),\quad \forall {\bf  y} \in \Xi_n.
$$
Denote by $\delta_{\bullet}^{{\bf p}}$ any best $r$-approximation of $\mu$,
given ${\bf p}$. (Again in the interest of readability, the $r$-dependence
of $\delta^{\bf p}_{\bullet}$ is made explicit by a subscript only when
necessary to avoid ambiguity.) An important special case of
${\bf p} \in\Pi_n$ is the uniform probability vector
${\bf u}_n=(1,\cdots,1)/n$. Best $r$-approximations of $\mu$, given ${\bf u}_n$,
will be referred to as best {\em uniform} $r$-approximations, and
denoted $\delta_{\bullet}^{{\bf u}_n}$. As in the case of prescribed
locations studied in Subsection~\ref{subloc}, the existence of best
$r$-approximations with prescribed weights follows from results in
Sections~3 and 4. Due to the nature of \eqref{0}, the proof of the
following theorem even is simpler than that of its counterpart, 
Theorem~\ref{th2}.

\begin{theorem}\label{th1}
Assume that $\mu\in\mathcal{P}_r$ for some $r\ge1$, and
$n\in\mathbb{N}$. 
For every ${\bf p} \in\Pi_n$, there exists a best $r$-approximation of
$\mu$, given ${\bf p}$. Moreover, $d_1\left(\delta_{{\bf x}}^{{\bf p}},\mu\right)=d_1\left(\delta^{\bf p}_{\bullet},\mu\right)$ if and only if, for every $i=1,\cdots,n$,
\begin{equation}\label{w1}
P_{ i-1}<P_{ i}\  \text{implies}\ x_{ i}\in
Q^{F_\mu}_{\frac{1}{2}(P_{ i-1}+P_{ i})},
\end{equation} 
and for $r>1$,
$d_r\left(\delta_{{\bf x}}^{{\bf p}},\mu\right)=d_r\left(\delta^{\bf p}_{\bullet},\mu\right)$
if and only if, for every $i=1,\cdots,n$,
\begin{equation}\label{w2}
P_{ i-1}<P_{ i}\ \text{implies}\ x_{ i}=\tau_r^{f_i},\ \text{where}\
f_i=F_{\mu}^{-1}\left|_{ [P_{ i-1},P_{ i} ]}\right. .
\end{equation}
\end{theorem}

\begin{proof}
As in the proof of Theorem~\ref{th2}, existence follows immediately,
once \eqref{w1} and \eqref{w2} are established. Labelling ${\bf P}$ as 
$$
P_{ i_0}=\cdots=P_{ i_1-1}<P_{ i_1}=\cdots=P_{ i_2-1}<P_{ i_2}=\cdots<
\cdots<P_{ i_{m-1}}=\cdots=P_{ i_m-1}
$$ 
with integers $j\le i_j\le
n+1$ for $1\le j\le m\le n$, and $i_0=0$, $i_m=n+1$, 
note that $d_r\left(\delta_{{\bf x}}^{{\bf
      p}},\mu\right)=d_r\left(\delta_{\overline{{\bf
        x}}}^{\overline{{\bf p}}},
\mu\right)$, where $\overline{{\bf x}}\in\Xi_m$ and $\overline{{\bf p}}\in\Pi_m$,
with $\overline{x}_{ j}=x_{ i_j}$, and $\overline{P}_{ j}=P_{ i_j}$
for $1\le j\le m$. Moreover, \eqref{w1} reduces to
$\overline{x}_{ j}\in Q_{\frac{1}{2}
  (\overline{P}_{ j-1}+\overline{P}_{ j} )}^{F_{\mu}}$ 
for all $1\le j\le m$, whereas \eqref{w2} reduces to
$\overline{x}_{ j}=\tau_r^{f_j}$ with $f_j=F_{\mu}^{-1}\left|_{
    [\overline{P}_{ j-1},\overline{P}_{ j} ]}\right.$.
Thus, to establish \eqref{w1} and \eqref{w2}, it can be assumed
w.o.l.g. that $P_{ i-1}<P_{ i}$ for all $i$.

Given ${\bf p}\in\Pi_n$, it is clear from $d_r (\delta_{{\bf x}}^{{\bf p}},\mu )^r=\sum_{i=1}^n \|x_{ i}-f_i \|_r^r$ that $d_r (\delta_{{\bf x}}^{{\bf p}},\mu )$ is minimal if and only if $ \|x_{ i}-f_i \|_r$ is minimal for every $i$. By Corollary~\ref{co2}, the latter is the case precisely if $x_{ i}\in Q_{\frac{1}{2} (P_{ i-1}+P_{ i} )}^{f_i^{-1}}=Q_{\frac{1}{2} (P_{ i-1}
+P_{ i} )}^{F_{\mu}}$  for $r=1$, and if $x_{ i}=\tau_r^{f_i}$ for $r>1$.
\end{proof}

\begin{rem}\label{re0}
(i) For $r=1$ and ${\bf p} ={\bf u}_n$, Theorem~\ref{th1} reduces to
\cite[Theorem~2.8]{BHM}. In particular,
$\frac{1}{n}\sum_{i=1}^n\delta_{F_{\mu}^{-1} (\frac{2i-1}{2n} )}$ is a
best uniform $1$-approximation of $\mu\in\mathcal{P}_1$. For $n=1$,
\eqref{w1} yields the well-known fact that $d_1(\delta_a,\mu)$ is
minimal if and only if $a\in\mathbb{R}$ is a {\em median} of 
$\mu$.

(ii) For $r=2$, if $\mu\in\mathcal{P}_2$ and ${\bf p} \in\Pi_n$ with
$p_{ i}>0$ for all $i$, then  by Remark~\ref{re1a}(i), the unique
best  $2$-approximation of $\mu$, given ${\bf p}$, is $\delta_{{\bf x}}^{{\bf p}}$ with
$x_{ i}=p_{ i}^{-1}\int_{P_{ i-1}}^{P_{ i}}F_{\mu}^{-1}(t)\, {\rm
  d}t$. In particular, $d_2 (\delta_a,\mu )$ is minimal precisely for
$a=\int_0^1F_{\mu}^{-1}(t)\, {\rm d}t$.
\end{rem}

\begin{example}\label{sq4}
Given $\mu\in\mathcal{P}_r$ and ${\bf p} \in\Pi_n$, Theorem~\ref{th1} can
also be utilized to minimize $d_r
(\sum_{i=1}^np_{ i}\delta_{x_{ i}},\mu )$ where ${\bf x} \in\mathbb{R}^n$ but
not necessarily ${\bf x} \in\Xi_n$. For instance, with $\mu={\rm Beta}(2,1)$
as in Example~\ref{sq} and ${\bf p} = (2/3,1/3 )$ as well as ${\bf q} = (1/3,2/3 )$,
for $r=1$, 
$$
\delta_{\bullet}^{{\bf p}}= {\textstyle\frac{2}{3}} \delta_{1/\sqrt{3}}+{\textstyle\frac{1}{3}}\delta_{\sqrt
{5/6}},\quad
\delta_{\bullet}^{{\bf q}} ={\textstyle\frac{1}{3}}\delta_{1/\sqrt{6}}+{\textstyle\frac{2}{3}}
\delta_{2/\sqrt
{6}}\, .
$$  
Since  $d_1 (\delta_{\bullet}^{{\bf p}},\mu )\thickapprox0.12154 >d_1
(\delta_{\bullet}^{{\bf q}} ,\mu )\thickapprox0.10677$, it follows, that
$\min_{{\bf x} \in\mathbb{R}^2}  d_1 (\frac{2}{3}\delta_{x_{ 1}}
+\frac{1}{3}\delta_{x_{ 2}},\mu )=d_1 (\delta_{\bullet}^{{\bf q}} ,\mu )$.
In general, this minimizing problem can be solved by applying
Theorem~\ref{th1} to $ (p_{ \sigma(1)},\cdots,p_{ \sigma(n)}
)\in\Pi_n$ for all permutations $\sigma$ of $ \{1,\cdots,n \}$. The
permutations yielding the minimal value may depend on $r$. Often, not
all $n!$ permutations $\sigma$ have to be considered. For instance, if
$F^{-1}_{\mu}$ is concave on $]0,1[$ as in the above example, then
only the (unique) non-decreasing rearrangement of $p$ is relevant.
\end{example}

Given $\mu\in\mathcal{P}_r$ and ${\bf p}_n\in\Pi_n$ for all $n$, it is
natural to ask whether $d_r (\delta_{\bullet}^{{\bf p}_n},\mu )\to0$ as
$n\to\infty$. As in the dual situation of Subsection~\ref{subloc},
this may or may not be the case, as illustrated by the 
following example.

\begin{example}\label{ex2}
Consider again the exponential distribution $\mu$ of
Example~\ref{expe}. By \eqref{w1}, the unique best uniform
$1$-approximation of $\mu$ is $\delta_{{\bf x}_n}^{{\bf u}_n}$ with
$x_{n,i}=F_{\mu}^{-1} (\frac{2i-1}{2n} )=\log\frac{2n}{2n-2i+1}$, for
every $n\in\mathbb{N}$ and $1\le i\le n$, 
and
$$
nd_1 (\delta_{\bullet}^{{\bf u}_n},\mu )=-2\sum\nolimits_{i=1}^ni\log\frac{2i-1}{2i}
+\log\frac{(2n)!}{2^{2n}n!n^n}=\frac{1}{4}\log n+\cO (1) \ \text{as}\
n\to\infty\, .
$$
By Remark~\ref{re0}(ii), the best uniform $2$-approximation of
$\mu$ is unique, namely $\delta_{{\bf y}_n}^{{\bf u}_n}$ with
$y_{n,i}=n\int_{(i-1)/n}^{i/n}F_{\mu}^{-1}(t) \, {\rm
  d}t=\log\displaystyle\frac{en(n-i)^{n-i}}{(n-i+1)^{n-i+1}}$, and  
$$
\sqrt{n}d_2 (\delta_{\bullet}^{{\bf u}_n},\mu )=\sqrt{n-\sum\nolimits_{i=1}^{n-1}i(i+1)
 (\log\frac{i}{i+1} )^2}=C_2+\cO ( n^{-1}) \quad  \text{as}\ n\to\infty\,
, $$
where $C_2^2=1+\sum_{i=1}^{\infty}\left(1-i(i+1)\left(\log\frac{i}{i+1}\right)^2
\right)
\thickapprox1.0803$. In fact, it can be shown that
$\lim_{n\to\infty}n^{1/r}d_r (\delta_{\bullet}^{{\bf u}_n},\mu )=C_r$
whenever $r>1$, with the appropriate $0<C_r<+\infty$. 
Thus $d_r (\delta_{\bullet}^{{\bf u}_n},\mu )\to0$ as $n\to\infty$, but the
rate of convergence evidently depends on $r$, and is slower than $
(n^{-1} )$. By contrast, consider ${\bf p}_n\in\Pi_n$ with $p_{n,i}=\frac{2^{i-1}}{2^n-1}$ for $1\le i\le n$.
Then $\lim_{n\to\infty}d_r (\delta_{\bullet}^{{\bf p}_n},\mu )=d_r (\nu,\mu
)>0$ with $\nu=\sum_{i=1}^{\infty}2^{-i}\delta_{a_i}$, and
$a_i=F_{\mu}^{-1} (3\cdot2^{-i-1} )$ if $r=1$ 
and $a_i=\tau_r^{F_{\mu}^{-1} |_{ [2^{-i},2^{-i+1} ]} .}$ if $r>1$.
\end{example}

Example~\ref{ex2} suggests a simple condition that may be imposed on
$({\bf p}_n)$, with ${\bf p}_n\in\Pi_n$ for every $n$, in order to guarantee that
$\lim_{n\to\infty}d_r (\delta_{\bullet}^{{\bf p}_n},\mu )=0$. The following
result is a counterpart of Theorem~\ref{thlo}. Due to the nature of
\eqref{0}, the proof is similar but not identical; recall that
$G^{F_{\mu}^{-1}}\subset\mathbb{I}$ for every $\mu\in\mathcal{P}$.

\begin{theorem}\label{th11}
Assume that $\mu\in\mathcal{P}_r$ for some $r\ge1$, and ${\bf p}_n\in\Pi_n$ for every $n\in\mathbb{N}$.
Then  $\lim_{n\to\infty}d_r (\delta_{\bullet}^{{\bf p}_n},\mu )=0$ if and
only if 
\begin{equation}\label{wei}
\lim\nolimits_{n\to\infty}\min\nolimits_{1\le i\le n}  |t-P_{n,i} |=0,\quad
\forall  t\in G^{F_{\mu}^{-1}}.
\end{equation} 
In particular, \eqref{wei} holds whenever $\lim_{n\to\infty}\max_{1\le i\le n}p_{n,i}=0$.
\end{theorem}

\begin{proof}
For every $n\in\mathbb{N}$, let $\delta^{{\bf p}_n}_{{\bf x}_n}$ be a best
$r$-approximation of $\mu$, given ${\bf p}_n$, and also
$f_n=F_{\delta_{{\bf x}_n}^{{\bf p}_n}}$, for convenience.

To see that \eqref{wei} is necessary, suppose that 
$$
\min\nolimits_{1\le i\le n_k}\left|t-P_{n_k,i}\right|\ge2\varepsilon,\quad
\forall  k\in\mathbb{N},
$$
for some $0<t<1,\ 0<\varepsilon<\min\{t,1-t\}$, and the appropriate
sequence $(n_k)$. (The other cases, $t=0$ and $t=1$, are analogous.)
Since $f_{n_k}^{-1}$ is constant on $ [t-\varepsilon,t+\varepsilon ]$
whereas $F_{\mu}^{-1}$ is not,  
$$
d_r (\delta_{\bullet}^{{\bf p}_{n_k}},\mu )^r=d_r (\delta_{{\bf x}_{n_k}}^{{\bf p}_{n_k
}}, \mu )^r\ge\min\nolimits_{c\in\mathbb{R}}
\int_{t-\varepsilon}^{t+\varepsilon} |F_{\mu}^{-1}(u)-c |^r{\rm
  d}u>0,\ k\in\mathbb{N},
$$
and hence $\limsup_{n\to\infty}d_r (\delta_{\bullet}^{{\bf p}_n},\mu )>0$.

To show that \eqref{wei} also is sufficient, assume that $t$ is a
continuity point of $F_{\mu}^{-1}$. If $t\in G^{F_{\mu}^{-1}}$ then,
given $\varepsilon>0$, there exist $t_1, t_2\in G^{F_{\mu}^{-1}}$ with
$ |F_{\mu}^{-1}(t_{1,2})-F_{\mu}^{-1}(t) |<\varepsilon$ and either
$t<t_1<t_2$ or $t_1<t_2<t$. Assume w.o.l.g. that $t<t_1<t_2$. (The
other case is similar.) By \eqref{wei}, $t<P_{n,i_n}<P_{n,i_n+1}<t_2$
for all sufficiently large $n$ and the appropriate $1\le i_n\le
n$. Since $f_n^{-1}$ is constant on $[P_{n,i_n},P_{n,i_n+1}]$ with a
value between $F_{\mu}^{-1} (P_{n,i_n} )\ge F_{\mu}^{-1}(t)$ and
$F_{\mu}^{-1} (P_{n,i_n+1} )\le F_{\mu}^{-1}(t)+\varepsilon$, clearly
$f_n^{-1}(t)\to F_{\mu}^{-1}(t)$. If, on the other hand, $t\notin
G^{F_{\mu}^{-1}}$, then let $]a,b[\, \subset\mathbb{I}$ be the largest
interval that contains $t$ but is disjoint from
$G^{F_{\mu}^{-1}}$. Assume w.o.l.g. that $0<a<b<1$. (The cases $a=0$
and $b=1$ are analogous.) Then $a,b\in G^{F_{\mu}^{-1}}$. Given
$\varepsilon>0$, since $F_{\mu}^{-1}-F_{\mu}^{-1}(t)\in
L^r(\mathbb{I})$, there exists $\delta>0$ such that $\int_A
|F_{\mu}^{-1}(u)-F_{\mu}^{-1}(t) |^r{\rm d}u<\varepsilon$ whenever
$\lambda(A)<\delta$. Let $i_n=\min \{1\le j\le n:\ P_{n,j}>t \}$. Note
that $P_{n,i_n-1}\le t<P_{n,i_n}$. If $a\le P_{n,i_n-1}<P_{n,i_n}\le
b$, then $f_n^{-1}(t)=F_{\mu}^{-1}(t)$. If $P_{n,i_n-1}<a$, then $
|P_{n,i_n-1}-a |=\min_{1\le i\le n}  |a-P_{n,i} |$, $\max
\{b,P_{n,i_n} \}-b\le\min_{1\le i\le n} |b-P_{n,i} |$, and 
$ (a-P_{n,i_n-1} )+\max \{b,P_{n,i_n} \}-b<\delta$ for all sufficiently
large $n$, by \eqref{wei}. Hence 
\begin{align*}
(t-a)   |F_{\mu}^{-1}(t) & -f_n^{-1}(t) |^r\le\int_{P_{n,i_n-1}}^{P_{n,i_n}}
  |F_{\mu}^{-1}(u)-f_n^{-1}(t) |^r{\rm d}u\\
& \le\int_{P_{n,i_n-1}}^{P_{n,i_n}} |F_{\mu}^{-1}(u)-F_{\mu}^{-1}(t) |^r
{\rm d}u\\
  &=\int_{P_{n,i_n-1}}^a |F_{\mu}^{-1}(u)-F_{\mu}^{-1}(t) |^r{\rm d}u+\int_b^{\max \{b,P_{n,i_n} \}} |F_{\mu}^{-1}(u)-F_{\mu}^{-1}(t) 
  |^r{\rm d}u<\varepsilon.
\end{align*}
For $P_{n,i_n}>b$, an analogous argument applies. In summary, $f_n^{-1}\to F_{\mu}^{-1}$ a.e. on $\mathbb{I}$, and the remaining argument is identical to the one in the proof of Theorem~\ref{thlo}.
\end{proof}

Since $\delta_{\bullet}^{{\bf p}}$ is a best approximation of
$\mu\in\mathcal{P}_r$ w.r.t. the metric $d_r$, given weights ${\bf p}$, it
is natural to ask whether $\delta_{\bullet}^{{\bf p}}$ reflects any basic
feature of $\mu$. Most basically perhaps, how is ${\rm supp}\
\delta_{\bullet}^{{\bf p}}$ related to ${\rm supp}\ \mu$ ? As the following
example shows, it may not be possible to guarantee ${\rm supp}\
\delta_{\bullet}^{{\bf p}} \subset{\rm supp}\ \mu$.

\begin{example}\label{Can}
Let $\mu$ be the Cantor probability measure, i.e., the
$\frac{\log2}{\log3}$-dimensional Hausdorff measure on the classical
Cantor middle third set. Using the fact that $Q_t^{F_{\mu}}$ is a
non-degenerate interval for every dyadic rational $0<t<1$, it is
readily seen that $\delta_{\bullet}^{{\bf u}_n}$ is not unique for any
$n\in\mathbb{N}$ whenever $r=1$. For instance, $\frac{1}{2}
(\delta_{1/5}+\delta_{4/5} )$ and $\frac{1}{2}
(\delta_{1/9}+\delta_{8/9} )$ both are best uniform $1$-approximations
of $\mu$, and $\{1/5,4/5\}\cap{\rm supp}\ \mu=\varnothing$ whereas
$\{1/9,8/9\}\subset{\rm supp}\ \mu$. For $r>1$, however,
$\delta_{\bullet}^{{\bf u}_n}$ always is unique. In fact,
$\delta_{\bullet}^{{\bf u}_{2^k}}$ even is independent of $r>1$, due to
symmetry, and ${\rm supp}\ \delta_{\bullet}^{{\bf u}_{2^k}}\cap\ {\rm supp}\
\mu=\varnothing$. For example, $\delta_{\bullet}^{{\bf u}_2}=\frac{1}{2}
(\delta_{1/6}+\delta_{5/6} )$ for all $r>1$, and $\{1/6,5/6\}\cap\ {\rm supp}\ \mu=\varnothing$.
\end{example}

To formalize the observations in Example~\ref{Can}, note that if
$\delta_{{\bf x}}^{{\bf p}}$ is a best $1$-approximation of $\mu$,
given ${\bf p}\in\Pi_n$,
then, by Theorem~\ref{th1}, $x_{ i}\in Q_{\frac{1}{2} (P_{ i-1}+P_{ i}
  )}^{F_{\mu}}$ whenever $P_{ i-1}<P_{ i}$. Since the endpoints of all
quantile sets $Q_t^{F_{\mu}}$ belong to ${\rm supp}\ \mu$, by
Proposition~\ref{pro2}, it is possible to choose ${\bf y} \in\Xi_n$ with $d_1
(\delta_{{\bf y}}^{{\bf p}},\mu )=d_1 (\delta_{{\bf x}}^{{\bf p}},\mu )$ and ${\rm supp}\
\delta_{{\bf y}}^{{\bf p}}\subset{\rm supp}\ \mu$. Similarly, if $r>1$, then
$x_{ i}=\tau_r^{f_i}$ with $f_i=F_{\mu}^{-1}\left|_{[P_{ i-1},P_{ i}]},\right.$ and consequently 
$x_{ i}\in [F_{\mu}^{-1} (P_{ i-1} ),F_{\mu}^{-1} (P_{ i}- ) ]$. By
Corollary~\ref{co2}(i), it follows that 
$$
\min\ {\rm supp}\ \mu=F_{\mu}^{-1} (P_{ 0}+ )\le x_{ i}\le
F_{\mu}^{-1} (P_{ n}- )=\max\ {\rm supp}\ \mu,\quad \forall  i=1,\cdots,n
\, .
$$
This establishes

\begin{prop}\label{Pro1}
Assume that $\mu\in\mathcal{P}_r$ for some $r\ge1$, and
$n\in\mathbb{N}$. If $r=1$ or ${\rm supp}\ \mu$ is connected, 
then there exists a best $r$-approximation $\delta_{\bullet}^{{\bf
    p}}$ of $\mu$, given ${\bf p} \in\Pi_n$, with ${\rm supp}\
\delta_{\bullet}^{{\bf p}} \subset {\rm supp}\ \mu$.
\end{prop}

Among the best approximations of $\mu$, given ${\bf p} \in\Pi_n$, the case of
{\em uniform} approximations, i.e., ${\bf p} ={\bf u}_n$, arguably is the most
important. In this case, Theorem~\ref{th11} has the following
corollary; see also \cite[Thm.2]{MJ}.

\begin{cor}\label{co1}
Assume that $\mu\in\mathcal{P}_r$ for some $r\ge1$, and $1\le s\le
r$. For every $n\in\mathbb{N}$, let $\delta_{\bullet,s}^{{\bf u}_n}$ be a
best uniform $s$-approximation of $\mu$. Then $\lim_{n\to\infty}d_r
(\delta_{\bullet,s}^{{\bf u}_n},\mu )=0$. In particular,
$\lim_{n\to\infty}d_r (\delta_{{\bf x}_n}^{{\bf u}_n},\mu )=0$ for
$x_{n,i}=F_{\mu}^{-1} (\frac{2i-1}{2n} )$, i.e., for
$\delta_{{\bf x}_n}^{{\bf u}_n}$ being one best uniform $1$-approximation of $\mu$.
\end{cor}

Despite its simplicity, Corollary \ref{co1} touches upon a recurring
theme and motivation of the present article, namely the surprising
versatility of best uniform $1$-approximations: Not only are they easy
to compute (by virtue of Theorem \ref{th1}) and hence preferable for
practical computations, but they also provide reasonably good uniform
$d_r$-approximations. In fact, beyond what Corollary \ref{co1}
asserts, they often even capture the precise rate of convergence of
$\bigl( d_r(\delta_{\bullet}^{{\bf u}_n} , \mu) \bigr)$; see, e.g., Example
\ref{exGal} below, and the discussion following Proposition \ref{Zador}.

\begin{rem}
For $r=s=2$, Corollary~\ref{co1} yields \cite[Thm.~3.6]{Ba}. In
\cite{Ba}, a {\em convex order} on $\mathcal{P}$ is considered, shown
to be preserved by best uniform $2$-approximations, and applied to the
numerical construction of martingales. We conjecture that best uniform
$r$-approximations preserve this order for all $r>1$. By contrast,
best (unconstrained) $2$-approximations, considered in Subsection~5.3
below, do not in general preserve the convex order; see \cite[Thm.2.1]{Ba}.
\end{rem}

The remainder of this subsection is devoted to a study of $d_r
(\delta^{{\bf u}_n}_{\bullet},\mu )$ as $n\to\infty$. Since best uniform
$r$-approximations may be hard to identify explicitly, we will also
consider {\em asymptotically} best uniform $r$-approximations. 
Formally, $ (\delta_{{\bf x}_n}^{{\bf u}_n} )$ with ${\bf x}_n\in\Xi_n$ for all
$n\in\mathbb{N}$ is a sequence of 
{\em asymptotically best uniform $r$-approximations} of
$\mu\in\mathcal{P}_r\setminus\left\{\delta_{{\bf x}}^{{\bf u}_i}:\
  i\in\mathbb{N},\ {\bf x} \in\Xi_i\right\}$ if
$$
\lim\nolimits_{n\to\infty}\frac{d_r (\delta_{{\bf x}_n}^{{\bf u}_n},\mu )}
{d_r (\delta_{\bullet}^{{\bf u}_n},\mu )}=1.
$$
To illustrate a possible behaviour of $ \bigl( d_r
(\delta^{{\bf u}_n}_{\bullet},\mu ) \bigr)$, as well as the practical relevance of asymptotically best uniform approximations, we first consider a simple example.

\begin{example}\label{eg5}
Let $\mu={\rm Beta}(2,1)$ as in Example~\ref{sq}. Theorem~\ref{th1}
yields a unique best uniform $r$-approximation of $\mu$ for every
$r\ge1$. For $r=1$, a short calculation shows that 
$$
nd_1
(\delta_{\bullet}^{{\bf u}_n},\mu )=\frac{1}{4}+\cO (n^{-1/2})\quad \text{as}\
n\to\infty,
$$
whereas for $r=2$, 
$$
nd_2 (\delta_{\bullet}^{{\bf u}_n},\mu
)=\frac{1}{4\sqrt{3}}\sqrt{\log n}+\cO (1) \quad \text{as}\
n\to\infty.
$$
For $1<r<2$, however, $\delta_{\bullet}^{{\bf u}_n}$ is not easy to
calculate explicitly. This not only makes the rate of convergence of $
\bigl( d_r (\delta_{\bullet}^{{\bf u}_n},\mu ) \bigr)$ hard to determine,
but it also emphasizes the need for simple {\em asymptotically} best
uniform approximations. In fact, Theorem~\ref{th6} below shows that,
for every $1\le r<2$, $\lim_{n\to\infty}nd_r
(\delta_{\bullet}^{{\bf u}_n},\mu )= \left( \frac{2^{1-2r}}
{(r+1)(2-r)} \right)^{1/r}$, and  $ (\delta_{{\bf x}_n}^{{\bf u}_n} )$, with
$x_{n,i}=\sqrt{\frac{2i-1}{2n}}$ for $1\le i\le n$, is a  sequence of
asymptotically best uniform $r$-approximations. By contrast, it turns
out that $\lim_{n\to\infty}n^{1/2+1/r}d_r (\delta_{\bullet}^{{\bf u}_n},\mu
)$ is finite and positive whenever $r>2$.
\end{example}

The observations in Example~\ref{eg5} are a special case of a general
principle: If the quantile function of $\mu\in\mathcal{P}_r$ is
absolutely continuous (and not constant), then
$ \bigl( nd_r (\delta_{\bullet}^{{\bf u}_n},\mu ) \bigr)$ converges to a
positive limit. This fact may be seen as an analogue, in the context
of best uniform approximations, of a classical result regarding best
approximations; cf. Proposition~\ref{Zador} below.

\begin{theorem}\label{th6}
  Assume that $\mu\in\mathcal{P}_r$ for some $r\ge1$. If $\mu^{-1}$ is absolutely continuous (w.r.t. $\lambda$)  then
  \begin{equation}\label{1}
    \lim\nolimits_{n\to\infty}nd_r (\delta_{\bullet}^{{\bf u}_n},\mu )=\frac{1}{2(r+1)
    ^{1/r}}\left(\int_{\mathbb{I}} \left( \frac{{\rm d}\mu^{-1}}{{\rm d}\lambda} \right)^r\right)^{1/r}.
  \end{equation}
Moreover, if $\frac{{\rm d}\mu^{-1}}{{\rm d}\lambda}\in
L^r(\mathbb{I})$ then $ (\delta_{{\bf x}_n}^{{\bf u}_n} )$, with
$x_{n,i}=F_{\mu}^{-1} (\frac{2i-1}{2n} )$ for $1\le i\le n$,  is a
sequence of asymptotically best uniform $r$-approximations of $\mu$, unless $\mu$ is degenerate, i.e., unless $\mu=\delta_a$ for some $a\in\mathbb{R}$.
\end{theorem}

\begin{proof}
For convenience, let $f=F_{\mu}^{-1}|_{]0,1[}$, as well as $J_{n,i}=
[\frac{i-1}{n},\frac{i}{n} ]$ and $x_{n,i}=f (\frac{2i-1}{2n} )$ for
$n\in\mathbb{N}$ and $1\le i\le n$. Note that the non-decreasing
function $f$ is absolutely continuous, by assumption. For the reader's
convenience, the following proof is divided into four steps: First,
\eqref{1} will be established assuming that $f$ has a $C^1$-extension
to $\mathbb{I};$ then \eqref{1} will be shown to hold in general,
regardless of whether both sides are finite (Step 2) or infinite (Step
3); finally, the assertion regarding asymptotically best uniform
approximations will be proved (Step 4).

\medskip

\noindent
{\it Step 1}. Assume $f$ can be extended to a $C^1$-function on
$\mathbb{I}$. Then  
\begin{align*}
  n^rd_r (\delta_{{\bf x}_n}^{{\bf u}_n},\mu )^r
  & = \ n^r\sum\nolimits_{i=1}^n\int_{J_{n,i}} |f(t)-x_{n,i} |^r{\rm d}t\le n^r\sum\nolimits_{i=1}^n (\max\nolimits_{J_{n,i}}f' )^r\int_{J_{n,i}} \left|t-\frac{2i-1}{2n}
   \right|^r{\rm d}t\\
  & =  \frac{1}{2^r(r+1)}\cdot\frac{1}{n}\sum\nolimits_{i=1}^n (\max\nolimits_{J_{n,i}}f' )^r.
\end{align*}
Since $ (f' )^r$ is Riemann integrable, $\lambda(J_{n,i})=1/n$, and similarly
$$
n^rd_r (\delta_{{\bf x}_n}^{{\bf u}_n},\mu )^r\ge\frac{1}{2^r(r+1)}\cdot\frac{1}{n}
\sum\nolimits_{i=1}^n (\min\nolimits_{J_{n,i}}f' )^r,
$$
it follows that $\lim_{n\to\infty}nd_r (\delta_{\bullet}^{{\bf u}_n},\mu )=\frac{1}{2(r+1)^{1/r}}
 (\int_{\mathbb{I}}f'(t)^r{\rm d}t )^{1/r}<+\infty$.
Moreover, $f'$ is uniformly continuous, hence given $\varepsilon>0$,
there exists $N\in\mathbb{N}$ such that 
$$ 
|f'(t)-f'(u) |\le\varepsilon,\quad \forall t,u\in\mathbb{I},\
|t-u|<\frac{1}{N}.
$$ 
Whenever $n\ge N$, therefore, the Mean Value Theorem yields 
$$
\left|f(t)-x_{n,i}-f' \left(\frac{2i-1}{2n} \right)\left(t-\frac{2i-1}{2n}\right)
\right|\le\varepsilon\left|t-\frac{2i-1}{2n}\right|,\quad \forall  t\in
J_{n,i},
$$
and consequently, with $y_{n,i}=\tau_r^{f|_{J_{n,i}}}$,
$$
| y_{n,i}-x_{n,i} |\le\frac
{\varepsilon}{n},\quad  \forall  1\le i\le n,
$$
by Proposition~\ref{quan}. It follows that 
$$
n^rd_r (\delta_{\bullet}^{{\bf u}_n},\delta_{{\bf x}_n}^{{\bf u}_n} )^r=n^rd_r (\delta
_{{\bf y}_n}^{{\bf u}_n},\delta_{{\bf x}_n}^{{\bf u}_n} )^r=n^r\sum\nolimits_{i=1}^n\int_{J_{n,i}}
 |y_{n,i}-x_{n,i} |^r\le \varepsilon^r\, ,
$$ 
and since $\varepsilon>0$ was arbitrary, 
$$
\limsup\nolimits_{n\to\infty} |n^rd_r (\delta_{\bullet}^{{\bf u}_n},\mu )^r-n^rd_r
 (\delta_{{\bf x}_n}^{{\bf u}_n},\mu )^r |\le\lim\nolimits_{n\to\infty}n^rd_r
 (\delta_{\bullet}^{{\bf u}_n},\delta_{{\bf x}_n}^{{\bf u}_n} )^r=0,
$$
which establishes \eqref{1}, with the same finite value on either side.

\medskip

\noindent
{\it Step 2.}
Let the non-decreasing and absolutely continuous function $f$ be
arbitrary, but assume that $f'\in L^r(\mathbb{I})$. Similarly, let 
$\widetilde{\mu}\in\mathcal{P}_r$ be such that $\widetilde{\mu}^{-1}$
is absolutely continuous, with
$\widetilde{f}:=F_{\widetilde{\mu}}^{-1}$ and $\widetilde{f}'\in
L^r(\mathbb{I})$. For $n\in\mathbb{N}$ and $1\le i\le n$, pick any
$t_{n,i}\in J_{n,i}$ and define $z_{n,i}=f(t_{n,i})$,
$\widetilde{z}_{n,i}=\widetilde{f}(t_{n,i})$. Below, it will be shown
that, for any $r\ge1$, 
\begin{equation}\label{18}
 |n^rd_r (\delta_{{\bf z}_n}^{{\bf u}_n},\mu )^r-n^rd_r
 (\delta_{\widetilde{{\bf z} }_n}
^{{\bf u}_n},\widetilde{\mu} )^r |\le2
 \|f'-\widetilde{f}' \|_r \|f'+\widetilde{f}' \|_r^{r-1},\quad \forall n\in\mathbb{N}.
\end{equation}
To see that \eqref{1} follows easily from \eqref{18}, at least under
the current assumption that $f'\in L^r(\mathbb{I})$, fix $r\ge1$ and
w.o.l.g. $0<\varepsilon<\|f'\|_r$. There exists
$\widetilde{\mu}\in\mathcal{P}_r$ such that $\widetilde{f}$ has a
$C^1$-extension to $\mathbb{I}$, and
$\|f'-\widetilde{f}'\|_r<\varepsilon$. With the appropriate $t_n$, let
$\delta_{{\bf z}_n}^{{\bf u}_n}$ be a best uniform $r$-approximation of $\mu$, and $\delta_{{\bf x}_n}^{{\bf u}_n}$ a best uniform $r$-approximation of $\widetilde{\mu}$. For all sufficiently large $n$, Step 1 and \eqref{18} yield
\[\begin{split}
  n^rd_r (\delta_{{\bf z}_n}^{{\bf u}_n},\mu )^r\le&\ n^rd_r (\delta_{{\bf x}_n}^{{\bf u}_n},\mu )^r\le n^rd_r (\delta_{\widetilde{{\bf x}_n}},\mu )^r+2\varepsilon (\varepsilon+2
  \|f'\|_r )^{r-1}\\
  \le&\ \frac{1}{2^r(1+r)} (\|f'\|_r+\varepsilon )^r+\varepsilon+2\varepsilon
   (\varepsilon+2\|f'\|_r )^{r-1},
\end{split}\]but also\[\begin{split}
  n^rd_r (\delta_{{\bf z}_n}^{{\bf u}_n},\mu )^r\ge&\ n^rd_r
  (\delta_{\widetilde{{\bf z}}_n}^{{\bf u}_n},\widetilde{\mu} )^r-2\varepsilon
   (\varepsilon+2 \|f' \|_r )^{r-1}\\
  \ge&\ \frac{1}{2^r(1+r)} ( \|f' \|_r-\varepsilon )^r-\varepsilon-2
  \varepsilon (\varepsilon+2 \|f' \|_r )^{r-1}.
\end{split}\]
Since $\varepsilon>0$ was arbitrary, this establishes \eqref{1}.

It remains to verify \eqref{18}, which only requires the elementary inequality, valid for all $r\ge1$,
\begin{equation}\label{15}
 |a^r-b^r |\le r |a-b | (a^{r-1}+b^{r-1} ),\quad \forall a, b \ge0,
\end{equation}
together with a repeated application of H\"{o}lder's inequality, as
follows: Note first that
\begin{equation}\label{17} 
d_r (\delta_{{\bf z}_n}^{{\bf u}_n},\mu )^r=\sum\nolimits_{i=1}^n\left(\int_{(i-1)/n}^{t_{n,i}}
\left(\int_t^{t_{n,i}}f'(u){\rm d}u\right)^r{\rm
  d}t+\int^{i/n}_{t_{n,i}}\left(\int^t_{t_{n,i}}f'(u){\rm
    d}u\right)^r{\rm d}t\right),
\end{equation} 
and consequently
\begin{align*}
  \left|d_r (\delta_{{\bf z}_n}^{{\bf u}_n},\mu )^r-d_r
    (\delta_{\widetilde{{\bf z}}_n}^
  {{\bf u}_n},\widetilde{\mu} )^r\right| \le \sum_{i=1}^n &\left(  \int_{(i-1)/n}^{t_{n,i}}\left|\left(\int_t^{t_{n,i}}
  f'(u)\, {\rm d}u\right)^r-\left(\int_t^{t_{n,i}}\widetilde{f}'(u)\,
  {\rm d}u\right)^r\right|{\rm d}t\right.\\
& +\left.\int^{i/n}_{t_{n,i}}\left|\left(\int^t_{t_{n,i}}f'(u)\, {\rm d}u\right)^r-\left(\int^t_{t_{n,i}}\widetilde{f}'(u)\, {\rm d}u\right)^r\right|{\rm d}t\right).
\end{align*}
With \eqref{15}, therefore,
\begin{align*}
& \int_{(i-1)/n}^{t_{n,i}}\left|\left(\int_t^{t_{n,i}}f'(u)\, {\rm
      d}u\right)^r-\left(\int_t^{t_{n,i}}\widetilde{f}'(u)\, {\rm d}u\right)^r\right|{\rm d}t\\
& \le r\int_{(i-1)/n}^{t_{n,i}}\left|\int_t^{t_{n,i}} \bigl(
  f'(u)-\widetilde{f}'(u) \bigr) \, {\rm
    d}u\right|\left(\left(\int_t^{t_{n,i}}f'(u){\rm
      d}u\right)^{r-1}+\left(\int_t^{t_{n,i}}\widetilde{f}'(u)\, {\rm
      d}u\right)^{r-1}\right)\, {\rm d}t\\
& \le 2r\int_{(i-1)/n}^{t_{n,i}}\left|\int_t^{t_{n,i}} \bigl(
  f'(u)-\widetilde{f}'(u)\bigr)\, {\rm d}u\right|\left(\int_t^{t_{n,i}}\left(f'(u)+\widetilde{f}'(u)\right){\rm d}u\right)^{r-1}{\rm d}t\\
& \le2r(a_i^-)^{1/r}(b_i^-)^{(r-1)/r},
\end{align*}
where, using H\"{o}lder's inequality again
\begin{align*}
a_i^- & =\int_{(i-1)/n}^{t_{n,i}}\left|\int_t^{t_{n,i}} \bigl(f'(u)-\widetilde{f}'
(u) \bigr)\, {\rm d}u\right|^r{\rm d}t\le\frac{1}{rn^r}\int_{(i-1)/n}^{t_{n,i}}\left|f'(t)-\widetilde{f}'(t)
\right|^r{\rm d}t,\\
b_i^- & =\int_{(i-1)/n}^{t_{n,i}}\left(\int_t^{t_{n,i}} \bigl( f'(u)+\widetilde{f}'
(u) \bigr) \, {\rm d}u\right)^r{\rm d}t\le\frac{1}{rn^r}\int_{(i-1)/n}^{t_{n,i}}\left(f'(t)+\widetilde{f}'(t)
\right)^r{\rm d}t.
\end{align*}
By a completely analogous argument,
$$
\int^{i/n}_{t_{n,i}}\left|\left(\int^t_{t_{n,i}}f'(u)\, {\rm
      d}u\right)^r-\left(\int^t_{t_{n,i}}\widetilde{f}'(u)\, {\rm
      d}u\right)^r\right|{\rm d}t\le2r (a_i^+ )^{1/r} (b_i^+
)^{(r-1)/r},\quad  \forall 1\le i\le n,
$$
where
\begin{align*}
a_i^+& =\int^{i/n}_{t_{n,i}}\left|\int^t_{t_{n,i}} \bigl( f'(u)-\widetilde{f}'(u)
 \bigr) \, {\rm d}u\right|^r{\rm d}t\le\frac{1}{rn^r}\int^{i/n}_{t_{n,i}}\left|f'(t)-\widetilde{f}'(t)\right|^r
{\rm d}t,\\
b_i^+ & =\int^{i/n}_{t_{n,i}}\left(\int^t_{t_{n,i}} \bigl( f'(u)+\widetilde{f}'(u)
\bigr) \, {\rm d}u\right)^r{\rm d}t\le\frac{1}{rn^r}\int^{i/n}_{t_{n,i}}\left(f'(t)+\widetilde{f}'(t)\right)^r
{\rm d}t.
\end{align*}
In summary, therefore, 
\begin{align*}
&n^r\left|d_r (\delta_{{\bf z}_n}^{{\bf u}_n},\mu )^r-d_r
  (\delta_{\widetilde{{\bf z}}_n}
^{{\bf u}_n},\widetilde{\mu} )^r\right|\le 2rn^r\sum\nolimits_{i=1}^n\left( (a_i^- )^{1/r} (b_i^- )^{(r-1)/r}+ (a_i^+ )
^{1/r} (b_i^+ )^{(r-1)/r}\right)\\
&\le2rn^r \left(\sum\nolimits_{i=1}^n (a_i^-+a_i^+ )\right)^{1/r}
\left( \sum\nolimits_{i=1}^n (b_i^
-+b_i^+ ) \right)^{(r-1)/r}\\
&\le 2 r n^r \left( \frac{1}{rn^r}\int_{\mathbb{I}} |f'(t)-\widetilde{f}'(t) |^r
{\rm d}t \right)^{1/r} \left( \frac{1}{rn^r}\int_{\mathbb{I}}
\bigl( f'(t)+\widetilde{f}'(t) \bigr)^r{\rm d}t \right)^{(r-1)/r}\\
&=2 \|f'-\widetilde{f}' \|_r \|f'+\widetilde{f}' \|_r^{r-1},
\end{align*}
which is just \eqref{18}.

\medskip

\noindent
{\it Step 3}. To establish \eqref{1} in case the value on the right is
$+\infty$, assume that $f'\notin L^r(\mathbb{I})$. For
$N\in\mathbb{N}$, let $g_N=\min\{f',N\}$ and, given $C>0$, choose $N$
so large that $\|g_N\|_r^r\ge2^r(1+r)C$. Let $\mu_N$ be a probability
measure with $ (F^{-1}_{\mu_N} )'=g_N$. By \eqref{17}, 
\begin{align*}
d_r (\delta_{{\bf z}_n}^{{\bf u}_n},\mu )^r & \ge\sum\nolimits_{i=1}^n \left( \int_{(i-1)/n}^{t_i}
 \left( \int_t^{t_i}g_N(u)\, {\rm d}u \right)^r{\rm
  d}t+\int_{t_i}^{i/n} \left(\int_{t_i}^tg_N(u)\, {\rm d}u \right)^r{\rm
  d}t \right) \\
& \ge d_r (\delta_{\bullet}^{{\bf u}_n},\mu_N )^r,
\end{align*}
and since Step 2 applies to $\mu_N$,
$$
\liminf\nolimits_{n\to\infty}n^rd_r (\delta_{\bullet}^{{\bf u}_n},\mu )^r\ge \lim\nolimits_{n\to\infty}n^rd_r (\delta_{\bullet}^{{\bf u}_n},\mu_N )^r=\frac{1}{2^r
(r+1)}\|g_N\|_r^r\ge C.
$$
As $C>0$ was arbitrary, $n^rd_r
(\delta_{\bullet}^{{\bf u}_n},\mu )^r\to+\infty$ whenever $f'\notin
L^r(\mathbb{I})$, i.e., \eqref{1} is valid in this case also.

\medskip

\noindent
{\it Step 4.} Finally, to prove the assertion regarding asymptotically
best uniform approximations, assume that $f'\in L^r(\mathbb{I})$. Note
that $ \|f' \|_r>0$ whenever $\mu\neq\delta_a$ for all
$a\in\mathbb{R}$. In this case, given $\varepsilon>0$, pick
$\widetilde{\mu}\in\mathcal{P}_r$ such that
$\widetilde{f}=F_{\widetilde{\mu}}^{-1}$ has a $C^1$-extension to
$\mathbb{I}$ and $ \|f'-\widetilde{f}' \|_r<\varepsilon$. 
By Step~1, $\lim_{n\to\infty}n^rd_r (\delta_{\widetilde{{\bf x}}_n}^{{\bf u}_n},\widetilde{\mu}
 )^r=\frac{ \|\widetilde{f}' \|_r^r}{2^r(r+1)}$, whereas Step~2 guarantees that
$\lim_{n\to\infty}n^rd_r (\delta_{\bullet}^{{\bf u}_n},\mu )^r=\frac{\|f'\|
_r^r}{2^r(r+1)}$, and \eqref{18} yields
$$
 |n^rd_r (\delta_{\widetilde{{\bf x}}_n}^{{\bf u}_n},\widetilde{\mu} )^r-n^r
d_r (\delta_{{\bf x}_n}^{{\bf u}_n},\mu )^r |\le2\varepsilon (1+2\|f'\|_r )
^{r-1}.
$$
Combining these three facts leads to
\begin{align*}
  \limsup\nolimits_{n\to\infty}\frac{d_r (\delta_{{\bf x}_n}^{{\bf u}_n},\mu )^r}{d_r (\delta
  _{\bullet}^{{\bf u}_n},\mu )^r} \le &\limsup\nolimits_{n\to\infty}
  \frac{n^rd_r (\delta_{\bullet}^{{\bf u}_n},\mu )^r+2\varepsilon (1+2 \|f'
   \|_r )^{r-1}}{2^{-r} \|f' \|_r^r(r+1)^{-1}} \\
  \le & \left(1+\frac{\varepsilon}{ \|f' \|_r} \right)^r+2^{r+1}(r+1)\varepsilon
  \frac{ (1+2 \|f' \|_r )^{r-1}}{ \|f' \|_r^r},
\end{align*}
as well as to an analogous lower bound for $\liminf_{n\to\infty}\frac{d_r (\delta_{{\bf x}_n}^{{\bf u}_n},\mu )^r}{d_r (\delta_
{\bullet}^{{\bf u}_n},\mu )^r}$. Since $\varepsilon>0$ was arbitrary, $\lim_{n\to\infty}\frac{d_r (\delta_{{\bf x}_n}^{{\bf u}_n},\mu )}{d_r (\delta_
{\bullet}^{{\bf u}_n},\mu )}=1$, i.e., $ (\delta_{{\bf x}_n}^{{\bf u}_n} )$ is a sequence of asymptotically best uniform $r$-approximations of $\mu$, as claimed.
\end{proof}

The following examples highlight the importance of the absolute
continuity and integrability assumptions, respectively, in
Theorem~\ref{th6}.

\begin{example}\label{eg3}
Let $\mu$ be the inverse of the Cantor probability measure in
Example~\ref{Can}. Explicitly, $\mu$ is purely atomic, with
$\mu\left( \{j2^{-m} \}\right)=3^{-m}$ for every $m\in\mathbb{N}$
and every odd $1\le j\le 2^m$. Note that
$F_{\mu}^{-1} |_{]0,1[}$ simply equals the classical Cantor
function, hence is continuous, in fact,
$\frac{\log2}{\log3}$-H\"{o}lder, and $\frac{{\rm d}\mu^{-1}}{{\rm
    d}\lambda}=0$ a.e.. While Theorem~\ref{th6}, if it did apply,
would seem to suggest that $\lim_{n\to\infty}nd_r (\delta_{\bullet}^{{\bf u}_n},\mu )=0$, a
detailed, elementary analysis shows that this is not the case. In
fact, $\bigl(n^{\alpha}d_r(\delta_{\bullet}^{{\bf u}_n},\mu ) \bigr)$ may
not converge to a finite positive limit for any $r\ge1$ and
$\alpha>0$. More specifically, let $\alpha_r=\frac{1}{r}+ (1-\frac{1}{r} )\frac{\log2}{\log3}$ for $r\ge1$. With this, $3^{\alpha_r}d_r (\delta_{\bullet}^{{\bf u}_{3n}},\mu )=d_r (\delta_{\bullet}
^{{\bf u}_n},\mu )$ for all $n$, and hence
$$
2^{-2+1/r} 3^{-2/r} \le\liminf\nolimits_{n\to\infty}n^{\alpha_r}
d_r (\delta
_{\bullet}^{{\bf u}_n},\mu )=\inf\nolimits_{n\in\mathbb{N}}n^{\alpha_r}d_r (\delta_{
\bullet}^{{\bf u}_n},\mu ),
$$
as well as
$$
\limsup\nolimits_{n\to\infty}n^{\alpha_r}d_r (\delta_{\bullet}^{{\bf u}_n},\mu )=\sup\nolimits_{n\in\mathbb{N}}n^{\alpha_r}d_r(\delta_{\bullet}^{{\bf u}_n},\mu)\le
2^{1/r} .
$$
We suspect the bounded sequence $\bigl( n^{\alpha_r}d_r (\delta_{\bullet}^{{\bf u}_n},\mu ) \bigr)$ 
to be divergent for every $r\ge1$. This illustrates that the conclusion of Theorem~\ref{th6} may fail if $\mu^{-1}$ is not absolutely continuous.
\end{example}

\begin{example}\label{exp}
The integrability assumption also is crucial (for the second
assertion) in Theorem~\ref{th6}. To see this, let $\mu$ be the
standard exponential distribution, where $\frac{{\rm d}\mu}{{\rm
    d}\lambda}\notin L^1(\mathbb{I})$, and \eqref{1} yields
$\lim_{n\to\infty}nd_r (\delta_{\bullet,r}^{{\bf u}_n},\mu )=+\infty$ for
all $r\ge1$, in perfect agreement with the observations in Example \ref{ex2}. Deduce
from a short calculation that 
$$
\sqrt{n}d_2 (\delta^{{\bf u}_n}_{\bullet,1},\mu )=D_2+\cO (n^{-1}) \quad \text{as}\
n\to\infty\, , 
$$ 
where $D_2^2=1+2\sum_{i=1}^{\infty}\left(1-i (1+\log\frac{\sqrt{4i^2-1}}{2i} )
\log\frac{2i+1}{2i-1}\right)\thickapprox1.1749$. Recall from
Example~\ref{ex2} that $\sqrt{n}d_2 (\delta_{\bullet,2}^{{\bf u}_n},\mu
)=C_2+\cO (n^{-1})$ with $C_2^2\thickapprox1.0803$. Thus while $
(\delta_{\bullet,1}^{{\bf u}_n} )$ identifies the correct rate of decay for
$ \bigl( d_2 (\delta_{\bullet,2}^{{\bf u}_n},\mu ) \bigr)$, namely
$\left(n^{-1/2}\right)$, it is {\em not} a sequence of asymptotically
best uniform $2$-approximations of $\mu$, since  
$$
\lim\nolimits_{n\to\infty}\frac{d_2 (\delta_{\bullet,1}^{{\bf u}_n},\mu )}{d_2\left
(\delta_{\bullet,2}^{{\bf u}_n},\mu\right)}=\frac{D_2}{C_2}>1\, . 
$$ 
Similarly, for any $r>1$ it can be shown that
$\lim_{n\to\infty}n^{1/r}d_r (\delta_{\bullet,1}^{{\bf u}_n},\mu )=D_r$ 
with the appropriate constant $D_r>C_r$, and $C_r$ as in Example~\ref{ex2}.\end{example}

\begin{example}\label{exGal}
Let $\mu$ be the standard normal distribution. By
\cite[Thm.1]{GHMR}, $ \bigl( d_2(\delta_{\bullet,2}^{{\bf u}_n},\mu) \bigr)$ decays
like $\left( n^{-1/2}(\log n)^{-1/2}\right) $ along the subsequence
$n=2^k$. Moreover, as pointed out in \cite[Rem.6]{GHMR}, best
uniform $1$-approximations $\delta_{\bullet,1}^{{\bf u}_n}$ converge to
$\mu$ as fast as $\delta_{\bullet,2}^{{\bf u}_n}$ do (w.r.t. $d_2$). In
fact, elementary computations confirm that the sequences
$$
 \bigl( n^{1/r}(\log n)^{\alpha_r}d_r (\delta_{\bullet,r}^{{\bf u}_n} , \mu
 ) \bigr),\  
\bigl( n^{1/r}(\log n)^{\alpha_r}d_r (\delta_{\bullet,1}^{{\bf u}_n} , \mu ) \bigr)
$$
are bounded above and below by positive constants, where
$\alpha_r=-\frac12$ if $r=1$, and $\alpha_r = \frac12 $ if $r>1$. 
Thus, best uniform $1$-approximations converge precisely as fast as do
best uniform $r$-approximations. Together with the previous example, this
suggests that $(\delta_{\bullet, 1}^{{\bf u}_n})$, though perhaps not a
sequence of asymptotically best uniform $r$-approximations, may
nevertheless often capture the correct rate of convergence of $\bigl(
d_r (\delta_{\bullet,r}^{{\bf u}_n}), \mu )\bigr)$, even when $\frac{{\rm
    d}\mu^{-1}}{{\rm d}\lambda}\not \in L^r (\mathbb{I})$.
\end{example}

\begin{rem}
For any (non-degenerate) $\mu\in\mathcal{P}$, \cite[Prop.~A.17]{BL}
asserts that $\mu^{-1}$ is absolutely continuous if and only if ${\rm
  supp}\ \mu$ is connected and $\frac{{\rm d}\mu_a}{{\rm d}\lambda}>0$
a.e. on ${\rm supp}\ \mu$, where $\mu_a$ is the absolutely continuous
part (w.r.t.\ $\lambda$) of $\mu;$ in this case, moreover, 
$\int_{\mathbb{I}}(\frac{{\rm d}\mu^{-1}}{{\rm
    d}\lambda})^r=\int_{{\rm supp}\ \mu}(\frac{{\rm d}\mu_a}{{\rm
    d}\lambda} )^{1-r}$.
\end{rem}

If $F_{\mu}^{-1}$ not even is continuous, then the decay of $ \bigl( d_r
(\delta_{\bullet}^{{\bf u}_n},\mu ) \bigr)$ may be less homogeneous than in
Example~\ref{eg3}. For instance, for the Cantor measure of Example~\ref{Can}, it is not
hard to see that, for any $r\ge1$, both numbers 
$$
\liminf\nolimits_{n\to\infty}n^{\frac{\log3}{\log2}}d_r (\delta_{\bullet}^{{\bf u}_n},
\mu )\ \text{and}\
\limsup\nolimits_{n\to\infty}n^{1/r}d_r (\delta_{\bullet}^{{\bf u}_n},\mu )
$$
are finite and positive. Thus, in general it cannot be expected that
for some $\alpha_r>0$, the sequence $ \bigl( n^{\alpha_r}d_r
(\delta_{\bullet}^{{\bf u}_n},\mu ) \bigr)$ is bounded below and above by
positive constants, let alone convergent. Still, it is possible to
identify a universal lower bound for $\bigl (d_r
  (\delta_{\bullet}^{{\bf u}_n},\mu )\bigr)$ with $\mu\in\mathcal{P}_r$:
But for trivial exceptions, this sequence never decays faster than 
$\left(n^{-1}\right)$.

\begin{theorem}\label{low}
 Assume that $\mu\in\mathcal{P}_r$ for some $r\ge1$.  Then
\begin{equation}\label{12}
\limsup\nolimits_{n\to\infty}nd_r (\delta_{\bullet}^{{\bf u}_n},\mu )>0,
\end{equation} 
unless $\mu=\delta_a$ for some $a\in\mathbb{R}$.
\end{theorem}

\begin{proof}
Denote $F_{\mu}^{-1}$ by $f$ for convenience, and for every $n\in\mathbb{N}$, let $a_i=f (\frac{2i-1}{4n} )$  and $b_i=f (\frac{2i-1}{4n+2} )$ for $1\le i\le2n$.
Then $b_1\le a_1\le b_2\le a_2\cdots\le b_{2n}\le a_{2n}\le b_{2n+1}$,
and 
\begin{align*}
  & 2nd_1  (\delta_{\bullet}^{{\bf u}_{2n}},\delta_{\bullet}^{{\bf u}_{2n+1}} )\\
  & =  2n\sum\nolimits_{i=1}^{2n}\left((a_i-b_i) \left(
      \frac{i}{2n+1}-\frac{i-1}{2n} \right)+(b_
  {i+1}-a_i) \left( \frac{i}{2n}-\frac{i}{2n+1} \right)\right)\\
  & =  \frac{1}{2n+1}\sum\nolimits_{i=1}^{2n} \bigl( (2n+1-i)(a_i-b_i)+i(b_{i+1}-a_i) \bigr)\\
  &\ge  \frac{1}{2n+1} \left(
    \sum\nolimits_{i=1}^{n} \!\!
    i(b_{i+1}-b_i)+ \! \sum\nolimits_{i=n+1}^{2n} \!
    \bigl( (2n   +1-2i)(a_i-b_i)+i(b_{i+1}-b_i) \bigr) \right) \\
  & =  \frac{1}{2n+1} \left(
    \sum\nolimits_{i=1}^{n} \!\!
    i(b_{i+1}-b_i)+\! \sum\nolimits_{i=n+1}^{2n} \! \bigl( (2n+1
  -i)(b_{i+1}-b_i)+(2i-2n-1)(b_{i+1}-a_i) \bigr) \right) \\
  & \ge  \frac{1}{2n+1} \left( \sum\nolimits_{i=1}^{n} \!\!
    i(b_{i+1}-b_i)+\! \sum\nolimits_{i=n+1}^{2n} \! (2n+1-i)
  (b_{i+1}-b_i) \right) \\
  & =
  \sum\nolimits_{i=n+2}^{2n+1}\frac{b_i}{2n+1}-\sum\nolimits_{i=1}^n\frac{b_i}{2n+1}\,
  .
  \end{align*}
Since $f$ is locally Riemann integrable on $]0,1[$, it follows that 
$$
\limsup\nolimits_{n\to\infty}2nd_1 (\delta_{\bullet}^{{\bf u}_{2n}},\delta_{\bullet}^{u
_{2n+1}} )\ge\int_0^{1/2} \bigl(f (t+{\textstyle\frac{1}{2}} )-f(t)
\bigr) \, {\rm d}t,
$$ 
and consequently
\begin{align*}
\limsup\nolimits_{n\to\infty}nd_r (\delta_{\bullet}^{{\bf u}_n},\mu )\ge &  \limsup\nolimits_{n\to\infty}nd_1 (\delta_{\bullet}^{{\bf u}_n},\mu )\ge{\textstyle\frac{1}{2}}
\limsup\nolimits_{n\to\infty}2nd_1 (\delta_{\bullet}^{{\bf u}_{2n}},\delta_{\bullet}^{{\bf u}_{2n
+1}} ) \\
\ge &\ {\textstyle\frac{1}{2}} \int_0^{1/2} \bigl(f
(t+{\textstyle\frac{1}{2}}  )-f(t) \bigr) \, {\rm d}t>0
\end{align*}
unless $f$ is constant, i.e., unless $\mu=\delta_a$ for some $a\in\mathbb{R}$.
\end{proof}

It is natural to ask whether Theorem~\ref{low} has a counterpart in
that there also exists a universal {\em upper} bound on $\bigl( d_r\left(\delta_{\bullet}^{{\bf u}_n},\mu\right) \bigr)$. In general, this is
not the case: As an immediate consequence of Theorem~\ref{slow} below,
given $r\ge1$ and any sequence $(a_n)$ of positive real numbers with
$\lim_{n\to\infty}a_n=0$, there exists $\mu\in\mathcal{P}_r$ such that
$d_r (\delta_{\bullet}^{{\bf u}_n},\mu )\ge a_n$, for all
$n\in\mathbb{N}$. Under additional assumptions, however, an upper
bound on $ \bigl( d_r (\delta_{\bullet}^{{\bf u}_n},\mu ) \bigr)$ can be established.

\begin{theorem}\label{th5-1}
Assume that $\mu\in\mathcal{P}_r$ for some $r\ge1$.
\begin{enumerate}
\item If $\mu\in\mathcal{P}_s$ with $s>r$ then $\lim_{n\to\infty}n^{1/r-1/s}d_r (\delta_{\bullet}^{{\bf u}_n},\mu )=0$.
\item If ${\rm supp}\ \mu$ is bounded then
  $\limsup_{n\to\infty}n^{1/r}d_r (\delta_{\bullet}^{{\bf u}_n},\mu
  )<+\infty$. 
\end{enumerate}
\end{theorem}

\begin{proof}
Again, for convenience, let $f=F_{\mu}^{-1}$, and $x_{n,i}=f
(\frac{2i-1}{2n} )$ for all $n\in\mathbb{N}$ and $1\le i\le n$. With
$t_0=F_{\mu}(0)$, assume w.o.l.g.\ that $0<t_0<1$. (The cases $t_0=0$
and $t_0=1$ are completely analogous.) 
Recall that $f$ is non-decreasing and right-continuous,
$(t-t_0)f(t)\ge0$ for all $t\in\mathbb{I}$, 
and $0\le f(t_0),\ -f(t_0-)<+\infty$. For all sufficiently large $n$, therefore,
\begin{align*}
d_r & ( \delta_{\bullet}^{{\bf u}_n}  ,\mu )^r  \le d_r (\delta_{{\bf x}_n}^{{\bf u}_n},\mu )^r
     =\sum\nolimits_{i=1}^n\int_{\frac{i-1}{n}}^{\frac{2i-1}{2n}} \left( \bigl(x_{n,i}-f(t)
      \bigr)^r+ \left( f \left( t+\frac{1}{2n} \right)-x_{n,i}
      \right)^r \right)  {\rm d}t\\
& \le \sum\nolimits_{i=1}^n\int_{\frac{i-1}{n}}^{\frac{2i-1}{2n}} \left(f \left(t+\frac{1}{2n}
 \right)-f(t) \right)^r\!\! {\rm d}t\le\int_{\frac{1}{4n}}^{1-\frac{1}{4n}} \left(f \left(t+\frac{1}{4n} \right)-f \left(t-
\frac{1}{4n} \right) \right)^r \!\! {\rm d}t\\
& = \int_{\frac{1}{4n}}^{t_0-\frac{1}{4n}} \left( \left|f \left(t-\frac{1}{4n} \right) \right|-
 \left|f \left( t+\frac{1}{4n} \right) \right| \right)^r \!\! {\rm d}t\\
&\quad +\int_{t_0-\frac{1}{4n}}^{t_0+\frac{1}{4n}} \left(f \left(t+\frac{1}{4n} \right)+ 
\left|f \left(t-\frac{1}{4n} \right) \right| \right)^r \!\! {\rm
d}t+\int_{t_0+\frac{1}{4n}}^{1-\frac{1}{4n}} \left(f
\left(t+\frac{1}{4n} \right)-f \left(t-\frac{1}{4n} \right) \right)^r
\!\! {\rm d}t\\
&\le \int_0^{t_0-\frac{1}{2n}} |f(t) |^r{\rm
  d}t-\int_{\frac{1}{2n}}^{t_0} |f(t) |^r{\rm
  d}t+2^{r-1}\int_{t_0-\frac{1}{2n}}^{t_0+\frac{1}{2n}} |f(t) |^r{\rm
  d}t\\
& \quad +\int_{t_0+\frac{1}{2n}}^1 |f(t) |^r{\rm
  d}t-\int_{t_0}^{1-\frac{1}{2n}} |f(t) |^r{\rm d}t\\
&= \ a_n+ (2^{r-1}-1 )b_n,
\end{align*}
where the numbers $a_n, b_n$ are given by
\[a_n=\int_0^{\frac{1}{2n}} |f(t) |^r{\rm d}t+\int^1_{1-\frac{1}{2n}}
|f(t) |^r{\rm d}t\quad 
\text{and}\quad
b_n=\int_{t_0-\frac{1}{2n}}^{t_0+\frac{1}{2n}} |f(t) |^r{\rm d}t,\] respectively.
Note that
$$
0\le nb_n\le\max \left\{f \left(t_0+\frac{1}{2n} \right),-f
  \left(t_0-\frac{1}{2n} \right) \right\}^r\to
\max \{f(t_0),-f(t_0-) \}^r\quad \text{as}\ n\to\infty,
$$
and hence $ (nb_n )$ is bounded.

{\rm (i)} If $\mu\in\mathcal{P}_s$ for some $s>r$ then, by virtue of
H\"{o}lder's inequality, 
$$
0\le a_n\le \left( \left(\int_0^{\frac{1}{2n}} |f(t)
|^s{\rm d}t \right)^{r/s}+ \left(\int^1_{1-\frac{1}{2n}} |f(t) |^s{\rm
  d}t\right)^{r/s} \right) 2^{r/s}n^{r/s-1},
$$
which shows that
$\lim_{n\to\infty}n^{1-r/s}a_n=0$. It follows that 
$$
0\le
n^{1-r/s}d_r\left(\delta_{\bullet}^{{\bf u}_n},\mu\right)^r\le n^{1-r/s}a_n+
(2^{r-1}-1 )n^{1-r/s}b_n\to0\quad  \text{as}\ n\to\infty,
$$
and hence
$\lim_{n\to\infty}n^{1/r-1/s}d_r\left(\delta_{\bullet}^{{\bf u}_n},\mu\right)=0$,
as claimed.

{\rm (ii)} If ${\rm supp}\ \mu$ is bounded then ${\rm
  esssup}_{\mathbb{I}}|f|$ is finite. In this case, $ (na_n )$ is
bounded, and so is $ \bigl(n^{1/r}d_r\left(\delta_{\bullet}^{{\bf u}_n},\mu\right) \bigr)$.
\end{proof}

\begin{rem}\label{re3}
(i) Boundedness of ${\rm supp}\ \mu$ is essential in
Theorem~\ref{th5-1}(ii), as evidenced, e.g., by Example~\ref{exp} for
$r=1$. Notice, however, that the conclusion of Theorem~\ref{th5-1}(ii) remains valid in this example whenever $r>1$.

(ii) If ${\rm supp}\ \mu$ is disconnected, and hence  $F_{\mu}^{-1}$
is discontinuous at some $0<t<1$, then there exists $(n_k)$ such that
$\langle n_kt\rangle \in [1/3,2/3 ]$ for all $k$. For all sufficiently
large $k$, therefore, 
\begin{align*}
  n_kd_r (\delta_{\bullet}^{{\bf u}_{n_k}},\mu )^r & \ge \
  n_k\min\nolimits_{c\in\mathbb{R}} \int^{ (\lfloor n_kt\rfloor +1 )/n_k}_{\lfloor n_kt\rfloor
  /n_k} |F_{\mu}^{-1}(s)-c |^r{\rm d}s\\
  & \ge  \min\nolimits_{c\in [F_{\mu}^{-1}(t-),F_{\mu}^{-1}(t) ]}
  {\textstyle \frac{1}{3}}
   \Bigl( \bigl(F_{\mu}^{-1}(t)-c \bigr)^r+ \bigl(c-F_{\mu}^{-1}(t-)
   \bigr)^r \Bigr) \\
  &\ge  2^{1-r}3^{-1} \bigl(F_{\mu}^{-1}(t)-F_{\mu}^{-1}(t-) \bigr)^r \, .
\end{align*}
Hence \eqref{12} can be strengthened to $\limsup_{n\to\infty}n^{1/r}d_r (\delta_{\bullet}^{{\bf u}_n},\mu )>0$ whenever ${\rm supp}\
\mu$ is disconnected. In fact, by Theorem~\ref{th5-1}(ii), $ (n^{-1/r}
)$ is the sharp upper rate of $ \bigl( d_r (\delta_{\bullet}^{{\bf u}_n},\mu
) \bigr)$ in case ${\rm supp}\ \mu$ is bounded and disconnected, a situation observed for instance for the Cantor measure of Example~~\ref{Can}.

(iii) By utilizing a {\em uniform decomposition approach}, a
multi-dimensional analogue of Theorem \ref{th5-1} is established in
\cite{C}, with the threshold rates of convergence depending both on $r$ and
on the dimension of the ambient (Euclidean) space.
\end{rem}

\subsection{Best approximations}\label{subbest}
This final subsection relates the results presented earlier to the
classical theory of best (unconstrained) approximations. 
Let $\mu\in\mathcal{P}_r$ for some $r\ge1$. Given $n\in\mathbb{N}$,
call the probability measure $\delta_{{\bf x}}^{{\bf p}}$ with ${\bf
  x}\in \Xi_n$ and ${\bf p} \in
\Pi_n$ a {\em best $r$-approximation} of $\mu$ if 
$$
d_r (\delta_{{\bf x}}^{{\bf p}},\mu )\le d_r (\delta_{\bf y}^{{\bf
    q}},\mu ),\quad \forall {\bf y} \in \Xi_n,\
{\bf q} \in\Pi_n.
$$
Denote by $\delta_{\bullet}^{\bullet,n}$ any best $r$-approximation of
$\mu$. (As before, the dependence of $\delta_{\bullet}^{\bullet,n}$ on
$r$ is made explicit by a subscript only where necessary to avoid ambiguities.)
It is well known that best $r$-approximations exist always.

\begin{prop}
\label{th4-0}{\rm\cite[Sec.~4.1]{GL}}.
Assume that $\mu\in\mathcal{P}_r$ for some $r\ge1$. For every
$n\in\mathbb{N}$, there exists a best $r$-approximation
$\delta_{\bullet}^{\bullet,n}$ of $\mu$. If $\#\ {\rm supp}\ \mu\ge n$
then $\#\ {\rm supp}\ \delta_{\bullet}^{\bullet,n}=n$.
\end{prop}

By combining Proposition~\ref{th4-0} with Theorems~\ref{th2} and
~\ref{th1}, a description of all best $r$-approximations is easily
established.

 \begin{theorem}\label{th3}
Assume that $\mu\in\mathcal{P}_r$ for some $r\ge1$, and
$n\in\mathbb{N}$. Let $\delta_{{\bf x}}^{{\bf p}}$ with ${\bf
  x}\in\Xi_n,\ {\bf p} \in\Pi_n$ be a best $r$-approximation of $\mu$. Then, for every $i=1,\cdots,n$,
\begin{enumerate}
\item $x_{ i}<x_{ i+1}$ implies $P_{ i}\in Q^{F^{-1}_{\mu}}_{\frac{1}{2} (x_{ i}+x_{ i+1} )}$; and
\item $P_{ i-1}<P_{ i}$ implies $x_{ i}\in Q^{F_\mu}_{\frac{1}{2} (P_{ i-1}+P_{ i} )}$ if $r=1$, or $x_{ i}=\tau_r^{f_i}$ with $f_i=F_{\mu}^{-1}\left|_{ [P_{ i-1},P_{ i} ]}\right.$ if $r>1$.
\end{enumerate}
Moreover, if $\#\ {\rm supp}\ \mu\le n$ then $\delta_{{\bf x}}^{{\bf p}}=\mu$, whereas
if $\#\ {\rm supp}\ \mu>n$ then $x_{ i}<x_{ i+1}$ and
$P_{ i-1}<P_{ i}$ for all $i=1,\cdots,n$.
\end{theorem}

\begin{proof}
Note that $\delta_{{\bf x}}^{{\bf p}}$ is both a best $r$-approximation of $\mu$,
given ${\bf p}$, and a best $r$-approximation of $\mu$, given ${\bf x}$, and thus
conclusions (i) and (ii) follow directly from Theorems~\ref{th2}
and~\ref{th1}. For the non-trivial case where $\#\ {\rm supp}\ \mu>n$,
Proposition~\ref{th4-0} implies that $\#\ {\rm supp}\ \delta_{{\bf x}}^{{\bf p}}=n$,
or equivalently, $x_{ i}<x_{ i+1}$ and $P_{ i-1}<P_{ i}$ for all $i=1,\cdots,n$.
\end{proof}

As an important special case of Theorem~\ref{th3}, assume that
$\mu\in\mathcal{P}_r$ is continuous. Then $Q^{F_{\mu}^{-1}}_a$ is a
singleton for every $a\in\mathbb{R}$, and Theorem~\ref{th3} asserts
that every best $1$-approximation $\delta_{{\bf x}}^{{\bf p}}$ of $\mu$ satisfies 
$$
F_{\mu} \left( \frac{x_{ i}+x_{ i+1}}{2} \right) =P_{ i},\ \text{and}\
F_{\mu}(x_{ i})=\frac{P_{ i-1}+P_{ i}}{2},\quad \forall i=1,\cdots,n,
$$ 
and hence in particular
\begin{equation}\label{cla1}
2F_{\mu}(x_{ i})=F_{\mu} \left(\frac{x_{ i-1}+x_{ i}}{2}
 \right)+F_{\mu} \left( \frac{x_{ i}+x_{ i+1}}{2} \right) ,\quad \forall i=1,\cdots,n.
\end{equation} 
Similarly, every best $2$-approximation of $\mu$ satisfies 
$$
F_{\mu} \left( \frac{x_{ i}+x_{ i+1}}{2} \right) =P_{ i},\ \text{and}\
(P_{ i}-P_{ i-1} )x_{ i}=\int_{P_{ i-1}}^{P_{ i}}F_{\mu}^{-1}(t)\, {\rm
  d}t,\quad \forall  i=1,\cdots,n,
$$
and consequently
\begin{equation}\label{cla2}
x_{ i}F_{\mu} \left( \frac{x_{ i}+x_{ i+1}}{2} \right)-x_{ i}F_{\mu}
 \left( \frac{x_{ i-1}+x_{ i}}{2} \right) = \int_{\frac{1}{2} (x_{ i-1}+x_{ i} )}
^{\frac{1}{2} (x_{ i}+x_{ i+1} )}x \, {\rm d}F_{\mu}(x),\quad \forall
i=1,\cdots,n.
\end{equation}
Note that \eqref{cla1} and \eqref{cla2} each yield $n$ equations for
$x_{ 1},\cdots,x_{ n}$. These equations are exactly the classical
optimality conditions, derived, e.g., in \cite[Sec.~5.2]{GL} by means
of Voronoi partitions.

\begin{example}\label{r}
Let $\mu=\frac{1}{2}\lambda |_{[0,1]} +\frac{1}{2}\delta_{1}$. While
$\mu$ is not continuous, and hence not directly amenable to the
classical conditions \eqref{cla1} and \eqref{cla2}, Theorem~\ref{th3} applies and yields, for instance, $\delta_{\bullet,r}^{\bullet,2}=\xi(r)\delta_{\xi(r)}+ \bigl(1-\xi(r) \bigr)
  \delta_{3\xi(r)}$ for all $r\ge1$, where $r\mapsto\xi(r)$ is smooth,
  decreasing, with $\xi(1)=\frac{1}{3}$, $\xi(2)=\frac{3-\sqrt{3}}{4}$,
  and $\lim_{r\to+\infty}\xi(r)=\frac{1}{4}$.
\end{example}

If (i) and (ii) in Theorem~\ref{th3} identify only a single
probability measure $\delta_{{\bf x}}^{{\bf p}}$ then the latter clearly is a best
$r$-approximation. In general, however, and unlike in
Theorems~\ref{th2} and \ref{th1}, the conditions of Theorem~\ref{th3}
are not sufficient, as the following example shows. Moreover, best
$r$-approximations in general are not unique, not even when $r>1$.

\begin{example}\label{three}
Consider $\mu=\frac{1}{3}\lambda_{[-1,1]}+\frac{1}{3}\delta_0$ and let
$n=2$. For $r=1$, Theorem~\ref{th3} identifies exactly three potential
best $1$-approximations $\delta_{{\bf x}_j}^{{\bf p}_j},\ j=1,2,3$, namely 
$$
{\textstyle {\bf x}_1= (-\frac{2}{3},0 ),\ {\bf p}_1= (\frac{2}{9},\frac{7}{9} ),\quad 
{\bf x}_2= (-\frac{1}{4},\frac{1}{4} ),\ {\bf p}_2= (\frac{1}{2},\frac{1}{2} ),\quad
{\bf x}_3= (0,\frac{2}{3} ),\ {\bf p}_3= (\frac{7}{9},\frac{2}{9} ).}
$$
It is clear from $d_1 (\delta_{{\bf x}_1}^{{\bf p}_1},\mu )=d_1 (\delta_{{\bf x}_3}^{{\bf p}_3},\mu )=\frac{2}
{9}<\frac{7}{24}=d_1 (\delta_{{\bf x}_2}^{{\bf p}_2},\mu )$ that the two
(non-symmetric) probability measure $\delta_{{\bf x}_1}^{{\bf p}_1},\
\delta_{{\bf x}_3}^{{\bf p}_3}$ are best $1$-approximations of $\mu$, whereas the
(symmetric) $\delta_{{\bf x}_2}^{{\bf p}_2}$ is not. Similarly, for $r=2$,
Theorem~\ref{th3} yields three candidates of which again only the two
non-symmetric ones turn out to be best $2$-approximations of $\mu$.
\end{example}

Since $d_r (\delta_{\bullet}^{\bullet,n},\mu )\le d_r
(\delta_{\bullet}^{{\bf u}_n},\mu )$ for every $\mu\in\mathcal{P}_r$ and
$n\in\mathbb{N}$, clearly $\lim_{n\to\infty}d_r
(\delta_{\bullet}^{\bullet,n},\mu )=0$. The rate of convergence of $
\bigl( d_r (\delta_{\bullet}^{\bullet,n},\mu ) \bigr)$ has been, and
continues to be, studied extensively; see, e.g.,
\cite{GL,GLP,KZ,K,Kr,PS} and the references therein. Arguably the
simplest situation occurs if $\mu\in\mathcal{P}_r$ has a non-trivial
absolutely continuous part and satisfies a mild moment condition. In
this case, $ \bigl( d_r (\delta_{\bullet}^{\bullet,n},\mu ) \bigr)$ decays like $ (n^{-1} )$ for every $r$.

\begin{prop}\label{Zador}{\rm\cite[Thm.~6.2]{GL}.}
Assume that $\mu\in\mathcal{P}_r$ for some $r\ge1$. If $\mu\in
\mathcal{P}_s$ with $s>r$ then
$$
   \lim\nolimits_{n\to\infty}nd_r (\delta_{\bullet}^{\bullet,n},\mu )=\frac{1}{2(r
   +1)^{1/r}} \left( \int_{\mathbb{R}}\left(\frac{{\rm d}\mu_a}{{\rm d}\lambda}\right)^{\frac{1}{r+1}} \right)^{\frac{r+1}{r}},
$$
where $\mu_a$ is the absolutely continuous part (w.r.t. $\lambda$)
  of $\mu$.
\end{prop}

It is instructive to compare Proposition~\ref{Zador} to
Theorem~\ref{th6}. To do so, assume that $\mu\in\mathcal{P}_s$ for
some $s>r$ and that $\mu^{-1}$ is absolutely continuous. Then
$\lim_{n\to\infty}nd_r (\delta_{\bullet}^{\bullet,n},\mu )$ and
$\lim_{n\to\infty}nd_r (\delta_{\bullet}^{{\bf u}_n},\mu )$ both are finite
and positive, provided that $\mu$ is non-singular and $\frac{{\rm
    d}\mu^{-1}}{{\rm d}\lambda}\in L^r(\mathbb{I})$. Thus $ \bigl( d_r
(\delta_{\bullet}^{\bullet,n},\mu ) \bigr)$ and $ \bigl( d_r
(\delta_{\bullet}^{{\bf u}_n},\mu ) \bigr)$ exhibit the same rate of decay,
namely $ (n^{-1} )$. Note that while the latter rate is a universal
{\em upper\/} bound on $ \bigl( d_r (\delta_{\bullet}^{\bullet,n},\mu
) \bigr)$, at least under the mild assumption that
$\mu\in\mathcal{P}_s$ for some $s>r$, it is a universal {\em lower\/}
bound on $ \bigl( d_r (\delta_{\bullet}^{{\bf u}_n},\mu ) \bigr)$, by
Theorem~\ref{low}. Even if both sequences decay at the same rate, however, $\lim_{n\to\infty}nd_r (\delta_{\bullet}^{\bullet,n},\mu )\le
\lim_{n\to\infty}nd_r (\delta_{\bullet}^{{\bf u}_n},\mu )$, and equality
holds only if either $\mu=\frac{1}{\lambda(I)}\lambda |_I $ for some
bounded, non-degenerate interval $I\subset\mathbb{R}$ or else
$\mu=\delta_a$ for some $a\in\mathbb{R}$. Thus only in the trivial
case of a (possibly degenerate) uniform distribution $\mu$ does $
(\delta_{\bullet}^{{\bf u}_n} )$ provide a sequence of asymptotically best
$r$-approximations of $\mu$ (as defined below).

\begin{example}\label{eg2}
Let $\mu$ be the exponential distribution of Example~\ref{expe}. For
$r=1$ and every $n\in\mathbb{N}$, \eqref{cla1} identifies a unique
best $1$-approximation $\delta_{{\bf x}_n}^{{\bf p}_n}$,
with 
$$
x_{n,i}=-2\log\frac{n+1-i}{\sqrt{n(n+1)}},\quad
P_{n,i}=\frac{i(2n+1-i)}{n(n+1)},\quad \forall i=1,\cdots,n \, .
$$ 
Here $\delta_{\bullet}^{\bullet,n}$ is unique, and $
nd_1(\delta_{\bullet}^{\bullet,n},\mu)=n\log (1+\frac{1}{n} )=1+\cO (
n^{-1}) $ as $ n\to\infty$,
in agreement with Proposition~\ref{Zador}. For comparison, recall from
Example~\ref{ex2} that $\lim_{n\to\infty}\frac{n}{\log n}d_1
(\delta_{\bullet}^{{\bf u}_n},\mu )=\frac{1}{4}$. For $r>1$, no explicit
expression seems to be known for $\delta_{\bullet}^{\bullet,n}$, not
even when $r=2$. However, in a sense made precise below, $
(\delta_{{\bf y}_n}^{\widetilde{{\bf p}}_n} )$ with
$$
y_{n,i}=(r+1)\log\frac{n+1}{n-i+1},\quad \widetilde{P}_{n,i}=1- \left(
  \frac{(n+1-i)(n-i)}{(n+1)^2}\right)^{\frac{r+1}{2}},\quad \forall
i=1,\cdots,n
$$ yields a sequence of asymptotically best $r$-approximations of $\mu$ for any $r>1$. 
\end{example}

Example~\ref{eg2} illustrates that even in very simple situations it
may be difficult to compute $\delta_{\bullet}^{\bullet,n}$
explicitly. Not least from a computational point of view, therefore,
it is natural to seek $r$-approximations that at least are optimal
{\em asymptotically}. Specifically, call $ (\delta_{{\bf x}_n}^{{\bf p}_n} )$ with
${\bf x}_n\in\Xi_n$, ${\bf p}_n\in\Pi_n$ for all $n\in\mathbb{N}$ a sequence of
{\em asymptotically best $r$-approximations} of $\mu\in\mathcal{P}_r$
with $\#\ {\rm supp}\ \mu=\infty$, if
$$
\lim\nolimits_{n\to\infty}\frac{d_r (\delta_{{\bf x}_n}^{{\bf p}_n},\mu )}{d_r (\delta
_{\bullet}^{\bullet,n},\mu )}=1.
$$
There exists a large literature on asymptotically best
approximations. Specifically, mild conditions (such as
$\mu\in\mathcal{P}_r$ being absolutely continuous with $\frac{{\rm
    d}\mu}{{\rm d}\lambda}$ H\"{o}lder continuous and positive on
$\overset{\circ}{{\rm supp}\ \mu}$, among others) have been
established which guarantee that $ (\delta_{{\bf x}_n}^{\bullet} )$ is a
sequence of asymptotically best approximations of $\mu$, where 
\begin{equation}\label{asym}
x_{n,i}=F_{\mu_r}^{-1}\left(\frac{i}{n+1} \right),\quad \forall
i=1,\cdots,n\,  , 
\end{equation} 
with $\frac{{\rm d}\mu_r}{{\rm d}\lambda}=\frac{\frac{{\rm d}\mu}{{\rm
d}\lambda}^{\frac{1}{r+1}}}{\int_{\mathbb{R}}\frac{{\rm
d}\mu}{{\rm d}\lambda}^{\frac{1}{r+1}}};$ see, e.g., \cite{L,PF}
and the references therein.

\begin{example}\label{Beta}
  Let $\mu={\rm Beta}(2,1)$ as in Examples~\ref{sq}
  and~\ref{eg5}. While for arbitrary $n\in\mathbb{N}$, the authors do
  not know of an explicit expression for
  $\delta_{\bullet}^{\bullet,n}$ for any $r\ge1$, \eqref{asym} yields
  a sequence $ (\delta_{{\bf x}_n}^{{\bf p}_n} )$ of asymptotically best
  $r$-approximations of $\mu$, with 
$$
x_{n,i}= \left( \frac{i}{n+1}\right)^{\frac{r+1}{r+2}},\quad
P_{n,i}=\frac{1}{4(n+1)^{\frac{2(r+1)}{r+2}}} \left( i^{\frac{r+1}{r+2}}+(i+1)^{
  \frac{r+1}{r+2}} \right)^2,\quad \forall  i=1,\cdots,n-1\,  ,
$$ 
and $x_{n,n}= (\frac{n}{n+1} )^{\frac{r+1}{r+2}}$. From this, a short
calculation leads to, for instance,
$$
nd_2 (\delta_{{\bf x}_n}^{{\bf p}_n},\mu
)=\frac{3}{8\sqrt{2}}+\cO (n^{-1})\quad \text{as}\ n\to\infty\, ,
$$ 
which is consistent with Proposition~\ref{Zador}.
\end{example}

If $\mu\in\mathcal{P}_s$ is singular then
Proposition~\ref{Zador} only yields
$\lim_{n\to\infty}nd_r (\delta_{\bullet}^{\bullet,n}, \mu )=0$. The
detailed analysis of $ \bigl( d_r (\delta_{\bullet}^{\bullet,n},\mu )\bigr)$ in
this case is an active research area, for which already a substantial
literature exists, notably for important classes of singular
probabilities such as self-similar and -conformal measures; see, e.g.,
\cite{GL,GL1,KZ,PB,PB1,RR}. A key notion in this context is the
so-called {\em quantization dimension} of $\mu\in\mathcal{P}_r$ of
order $r$, defined as
$$
D_r(\mu)=\lim\nolimits_{n\to\infty}\frac{\log n}{-\log d_r
  (\delta_{\bullet}^{\bullet,n},\mu )},
$$
provided that this limit exists. For instance, Proposition~\ref{Zador}
implies that $D_r(\mu)=1$ whenever $\mu_a\neq0$. The relations of
$D_r(\mu)$ to various other concepts of dimension have attracted
considerable attention \cite{GL,KZ,PB,RR}.

\begin{example}\label{Canb}
For the Cantor measure $\mu$ of Example~\ref{Can}, \cite[Cor.~4.7 and
Rem.~6.1]{Kr0} show that, for every $r>1$,
$$
0<\liminf\nolimits_{n\to\infty}n^{\log3/\log2}d_r(\delta_{\bullet}^
{\bullet,n},\mu)
<\limsup\nolimits_{n\to\infty}n^{\log3/\log2}d_r (\delta_{\bullet}^{\bullet,n},\mu
 )<+\infty.
$$
From this, it is clear that $D_r(\mu)=\frac{\log2}{\log3}$, which is
independent of $r$ and coincides with the Hausdorff dimension of ${\rm
  supp}\ \mu$.
\end{example}

\begin{example}\label{Cani}
Let $\mu$ be the inverse Cantor measure of Example~\ref{eg3}. Note
that $\mu$ is not a self-similar, and hence the classical results for
self-similar probabilities do not apply. Still, $\mu$ is the unique
fixed point of a contraction on $\mathcal{P}_1$, namely
$\nu\mapsto\frac{1}{3} (\nu\circ T_1^{-1}+\delta_{1/2}+\nu\circ
T_2^{-1} )$, with the similarities $T_1(x)=\frac{1}{2}x$ and
$T_2(x)=\frac{1}{2}(1+x)$. This property enables a fairly complete
analysis of $ \bigl( d_r (\delta_{\bullet}^{\bullet,n},\mu ) \bigr)$
which the authors intend to present elsewhere. Specifically, with
$\beta_r= (1-\frac{1}{r} )+\frac{1}{r}\frac{\log3}{\log2}$, it can be
shown that, for every $r\ge1$, the numbers
$\liminf_{n\to\infty}n^{\beta_r}d_r (\delta_{\bullet}^{\bullet,n},\mu
)$ and $\limsup_{n\to\infty}n^{\beta_r}d_r
(\delta_{\bullet}^{\bullet,n},\mu )$ both are finite and positive. In
particular, $D_r(\mu)=\beta_r^{-1}$. Note that, unlike in the previous
example, $D_r(\mu)$ depends on $r$, and $\frac{\log2}{\log3}\le
D_r(\mu)<1$. Thus $D_r(\mu)$ is larger than $0$, the Hausdorff
dimension of $\mu$, but smaller than $1$, the Hausdorff dimension of 
${\rm supp}\ \mu=\mathbb{I}$.
\end{example}

Proposition~\ref{Zador} guarantees that under a mild moment condition,
$\bigl(d_r (\delta_{\bullet}^{\bullet,n},\mu )\bigr)$ decays at least
like $\bigl(n^{-1}\bigr)$, and in fact may decay faster, as
Examples~\ref{Canb} and~\ref{Cani} illustrate. Even for purely atomic
$\mu$, however, the decay of $ \bigl( d_r (\delta_{\bullet}^{\bullet,n},\mu )\bigr)$ can be
arbitrarily slow. This final observation, a refinement of
\cite[Ex.6.4]{GL}, uses the following simple calculus fact; cf.\ also
\cite[Thm.3.3]{BL}.

\begin{prop}\label{pro3}
Given any sequence $ (a_n )$ of non-negative real numbers with
$\lim_{n\to\infty}a_n=0$, there exists a decreasing sequence $ (b_n )$
with $\lim_{n\to\infty}b_n=0$ such that $ (b_{n}-b_{n+1} )$ is
decreasing also, and $b_n\ge a_n$ for all $n$.
\end{prop}

\begin{theorem}\label{slow}
Given $r\ge1$ and any sequence $ (a_n )$ of non-negative real
numbers with $\lim_{n\to\infty}a_n=0$, there exists 
$\mu\in\mathcal{P}_r$ such that $d_r (\delta_{\bullet}^{\bullet,n},\mu )\ge a_n$ for every $n\in\mathbb{N}$.
\end{theorem}

\begin{proof}
In view of Proposition~\ref{pro3}, assume w.o.l.g.\ that $ (a_n )$ and
$ (a_n^r-a_{n+1}^r )$ both are decreasing. Pick $a_0>a_1$ such that
$a_0^r-a_1^r>a_1^r-a_2^r$, and let
$c_r=\sum_{k=1}^{\infty}2^{-(k-1)r}$ $ (a_{k-1}^r-a_k^r )$. Note that
$c_r$ is finite and positive. Consider
$\mu=\sum_{k=1}^{\infty}p_k\delta_{x_k}$, where $p_k=c_r^{-1}
2^{-(k-1)r} (a_{k-1}^r-a_k^r )$ and $x_k=3\cdot2^{k-1}c_r^{1/r}$ for
all $k\in\mathbb{N}$. Since
$\sum_{k=1}^{\infty}p_kx_k^r=3^ra_0^r<+\infty$, clearly
$\mu\in\mathcal{P}_r$. For every $n\in\mathbb{N}$, 
define $K_n\subset\mathbb{N}$ as
$$
K_n= \{k\in\mathbb{N}:\ {\rm supp}\ \delta_{\bullet}^{\bullet,n}\
\cap\  [2^kc_r^{1/r},2^{k+1}c_r^{1/r} [\: =\varnothing \}\, .
$$ 
Since $\#\ {\rm supp}\  \delta_{\bullet}^{\bullet,n}\le n$ and the
intervals $  [2^kc_r^{1/r},2^{k+1}c_r^{1/r} [$, $k\in\mathbb{N}$,
are disjoint, $\#\ (\mathbb{N}\setminus K_n)\le n$. Moreover, 
$$
\min\nolimits_{y\in{\rm supp} \, \delta_{\bullet}^{\bullet,n}} |x_k-y|^r\ge2^{(k-1)r}c_r\
\text{for every}\ k\in K_n\, .
$$ 
Recall from \cite[(ii), p.1847]{CGI} that $d_r
(\delta_{\bullet}^{\bullet,n},\mu )^r=\int_{\mathbb{R}}\min_{y\in{\rm
    supp}\ \delta_{\bullet}^{\bullet,n}} |x-y|^r{\rm d}\mu(x)$; see
also \cite[Lemma~3.1]{GL}. It follows that, for every
$n\in\mathbb{N}$, 
$$
d_r (\delta_{\bullet}^{\bullet,n},\mu
)^r=\sum\nolimits_{k=1}^{\infty}p_k\min
\nolimits_{y\in{\rm supp}\ \delta_{\bullet}^{\bullet,n}} |x_k-y|^r\ge\sum\nolimits_{k\in
  K_n}p_k2^{(k-1)r}c_r=\sum\nolimits_{k\in K_n} (a_{k-1}^r-a_k^r ).
$$ 
Moreover, recall that $ (a_{n-1}^r-a_n^r )$ is decreasing, and $\#\
(\mathbb{N}\setminus K_n)\le n$. Thus 
$$
d_r (\delta_{\bullet}^{\bullet,n},\mu )^r\ge\sum\nolimits_{k=n+1}^{\infty} (a_{k-
1}^r-a_k^r )=a_n^r,
$$
and hence $d_r (\delta_{\bullet}^{\bullet,n},\mu )\ge a_n$ for every $n\in\mathbb{N}$.
\end{proof}

\subsection*{Acknowledgements}

The first author was supported in part by a Pacific Institute for the Mathematical
Sciences (PIMS) Graduate Scholarship and a Josephine Mitchell Graduate
Scholarship, both at the University of Alberta. The second author was supported by an NSERC Discovery
Grant. Both authors are grateful to F.\ Dai for helpful
conversations. They also like to thank an anonymous referee for
bringing to their attention the important reference \cite{MJ}, and for
providing several excellent comments that helped improve the exposition.

\end{document}